%% file: neumann_arxiv.tex
\documentclass[a4paper]{amsart}
\usepackage{a4}
\usepackage{amsmath}
\usepackage{amssymb}
\usepackage[applemac]{inputenc}
\usepackage{enumerate}
\usepackage{graphicx}
\usepackage{color}
\usepackage{multirow}
\usepackage{subcaption}
\usepackage{array}
\newcommand{\Tr}{ \mbox{{\rm Tr}}}

\newcommand{\uo}{{\overline u}}
\newcommand{\uu}{{\underline u}}

\newcommand{\PP}{ \mathbb{P}}

\newcommand{\EE}{{\mathbb E}}

\newcommand{\ds}{\displaystyle}

\newcommand{\T}{{\mathbb{T}}}

\PassOptionsToPackage{normalem}{ulem}
\usepackage{ulem}

\newtheorem{remark}{\textbf{Remark}}[section]

\newtheorem{lemma}{\textbf{Lemma}}[section]
\newtheorem{theorem}{\textbf{Theorem}}[section]

\newtheorem{proposition}{\textbf{Proposition}}[section]

\newtheorem{definition}{\textbf{Definition}}[section]

\numberwithin{equation}{section}

\title[A semi-Lagrangian scheme for HJB equations with oblique boundary conditions]{A semi-Lagrangian scheme for Hamilton-Jacobi-Bellman equations with oblique  boundary conditions} 

\author{Elisa Calzola \and Elisabetta Carlini \and Xavier Dupuis \and Francisco J. Silva}

\thanks{Dipartimento di Matematica  Guido Castelnuovo, Sapienza  Universit{\`a} di Roma, piazza Aldo Moro}
 
\thanks{Institut de Math\'ematiques de Bourgogne, UMR 5584 CNRS, Universit\'e Bourgogne Franche-Comt\'e, 21000 Dijon, France. 
              \email{xavier.dupuis@u-bourgogne.fr}}
              
\thanks{Institut de recherche XLIM-DMI, UMR-CNRS 7252, Facult\'e des Sciences et Techniques, 
Universit\'e de Limoges, 87060 Limoges, France \email{francisco.silva@unilim.fr}}
\input mymacro

\begin{document}

\begin{abstract}
We investigate in this work a fully-discrete semi-Lagrangian approximation of second order possibly degenerate Hamilton-Jacobi-Bellman (HJB) equations on a  bounded domain $\OO\subset \RR^N$ with oblique boundary conditions. These equations appear naturally in the study of optimal control of diffusion processes with oblique reflection at the boundary of the domain. 

The proposed scheme is shown to satisfy a consistency type property, it is monotone and stable. Our main result is the convergence of the numerical solution towards the unique viscosity solution of the HJB equation. The convergence result holds under the same asymptotic relation between the time and space discretization steps as in the classical setting for semi-Lagrangian schemes  on $\OO=\RR^{N}$. We present  some numerical results that confirm the numerical convergence of the scheme. \\[0.5ex]

{\small
\noindent {\bf AMS-Subject Classification:}  49L25 \and 65M12 \and  35K55 \and 49L20 \\[0.5ex]

}

\end{abstract}

\maketitle
\section{Introduction} \label{Sect_Intro} 

In this work we deal  with the numerical approximation of the following parabolic Hamilton-Jacobi-Bellman (HJB) equation 
\be\label{HJB_principal_equation}\ba{rcl}  \partial_t u  + H\left(t,x, Du , D^{2}u \right)&=&0 \hspace{0.3cm} \mbox{in } \; (0,T]\times \OO, \\[6pt]
L(t, x, D u)&=&0  \hspace{0.3cm} \mbox{on } \; (0,T] \times \partial \OO,\\[6pt]
u(0,x)&=&\Psi(x)  \hspace{0.4cm} \mbox{in  } \ov{\OO}.
\ea 
\ee
In the system above,  $T>0$, $\OO\subset \RR^N$ is a nonempty smooth bounded open set and $H$ and $L$ are nonlinear functions having the form   
\begin{eqnarray}
H(t,x,p,M)&=&\sup_{a\in A}\left\{-\frac{1}{2} \Tr\left( \sigma(t,x,a)\sigma(t,x,a)^\top M \right) - \langle \mu (t,x,a),p \rangle - f(t,x,a)\right\},\label{H_sup}\\
L(t,x,p) &=&\sup_{b\in B}\left\{\langle \gamma(x,b),p \rangle - g(t,x,b)\right\}, \label{B_sup}
\end{eqnarray} 
where $A\subset \RR^{N_{A}}$ and $B\subset \RR^{N_{B}}$ are nonempty compact sets, $\sigma: [0,T]\times \ov{\OO} \times A \to \RR^{N \times N_{\sigma}}$, with $1\leq N_{\sigma}\leq N$,  $\mu:  [0,T]\times \ov{\OO} \times A \to \RR^{N}$, $f:  [0,T]\times \ov{\OO} \times A \to \RR$, $\gamma: \partial \OO \times \V \to \RR^N$, with $\V\subseteq \RR^{N_{B}}$ being an open set containing $B$,  $g: [0,T]\times \partial{\OO}\times B \to \RR$, and $\Psi:\ov{\OO}\to \RR$.

If $A=\{a\}$ and $B=\{b\}$, for some $a\in \RR^{N_{A}}$ and $b\in \RR^{N_{B}}$, and $\gamma(x,b)=n(x)$, with $n(x)$ being the unit outward normal vector  to $\ov{\OO}$ at $x\in \partial \OO$, then \eqref{HJB_principal_equation} reduces to a standard linear parabolic equation with Neumann boundary conditions. In the general case, and after a simple change of the time variable in order to write \eqref{HJB_principal_equation} in backward form, the HJB equation \eqref{HJB_principal_equation} appears in the study of optimal control of diffusion processes with controlled reflection on the boundary $\partial \OO$ (see e.g. \cite{Lions1985} for the first order case, i.e. $\sigma\equiv 0$, and \cite{MR709164,MR2421330} for the general  case). Since the HJB equation  \eqref{HJB_principal_equation} is possibly degenerate parabolic,  one cannot expect the existence of classical solutions and we have to rely on the notion of viscosity solution (see e.g. \cite{CrandallIshiiLions92}). Moreover, as it has been noticed in \cite{MR667669,Lions1985}, in general the boundary condition in \eqref{HJB_principal_equation}  does not hold in the pointwise sense and we have to consider a suitable weak formulation of it. We refer the reader to \cite{Lions1985,MR1090787} and \cite{CrandallIshiiLions92,Barles1993,MR1685618,MR2070626,MR2399437}, respectively, for well-posedness results for HJB equations with oblique boundary condition in the first and second order cases. 

The study of the numerical approximation of solutions to HJB and, more generally, fully nonlinear second order Partial Differential Equations (PDEs), has made important progress over the last few decades. Most of the related literature consider the case where $\OO=\RR^N$, or where a Dirichlet boundary condition is imposed on the boundary $\partial \OO$.   We refer the reader to \cite{falconeferretilibro,MR3049920,MR3653852}  and the references therein for  the state of the art on this topic. By contrast, the numerical approximation of solutions to \eqref{HJB_principal_equation} has been much less explored. Indeed, to the best of our knowledge only the methods in \cite{MR1181342,MR2034614} can be applied to approximate \eqref{HJB_principal_equation} in the particular  first order case ($\sigma\equiv 0$). Moreover, in \cite{MR1181342}, where a finite difference scheme is proposed,  the function defining the boundary condition has the particular form $L(t,x,p,b)= \langle n(x), p\rangle$. On the other hand, both references consider Hamiltonians which are not necessarily convex with respect to $p$. Let us also mention the reference \cite{AF12}, where, in the context of mean curvature motion with nonlinear Neumann boundary conditions, the authors propose a discretization that combines a  Semi-Lagrangian (SL) scheme in the main part of the domain with a finite difference scheme near the boundary.

The main purpose of this article is to provide a consistent, stable, monotone and convergent SL scheme to approximate the unique viscosity solution to \eqref{HJB_principal_equation}. By the results in \cite{Barles1993}, the latter is well-posed in $C([0,T]\times\ov{\OO})$ under the assumptions in Sect.~\ref{prelim_results} below. Semi-Lagrangian schemes to approximate the solution  to \eqref{HJB_principal_equation} when  $\OO=\RR^N$ (see e.g. \cite{CamFal95,MR3042570}) can be derived from the optimal control interpretation of \eqref{HJB_principal_equation} and a suitable discretization of the underlying controlled trajectories. These schemes enjoy the feature that they are explicit and stable under an inverse Courant-Friedrichs-Lewy (CFL) condition and, consequentely, they allow large time steps. A second important feature is that they permit a simple treatement of the possibly degenerate second order term in $H$.  The scheme that we propose for $\OO\neq \RR^N$ preserves these two properties and seems to be the first convergent scheme to approximate \eqref{HJB_principal_equation} with the  rather general  asumptions  in Sect.~\ref{prelim_results}.  In particular, our results cover the stochastic and degenerate case. Consequently, from the stochastic control point of view, our scheme allows to approximate the so-called value function of the optimal control of a controlled diffusion process with possibly oblique reflection on the boundary $\partial \OO$ (see \cite{MR2421330}). The main difficulty in devising such a scheme is to be able to obtain a consistency type property   at points in the space grid which are near the boundary $\partial \OO$ while maintaining the stability. This is achieved by considering a discretization of the underlying controlled diffusion which suitably emulates its reflection at the boundary in the continuous case. We refer the reader to \cite{Milstein96} for a related construction of a semi-discrete in time approximation of a second order non-degenerate linear parabolic equation.

The remainder of this paper is structured as follows. In Sect.~\ref{prelim_results} we state our assumptions, we recall the notion of viscosity solution to  \eqref{HJB_principal_equation} and the well-posedness result. In Sect.~\ref{fully_discrete_scheme} we provide the SL scheme as well as its probabilistic interpretation (in the spirit of \cite{Milstein96}). The latter will play an important role in Sect.~\ref{properties_fully_discrete_scheme}, which is devoted to show a consistency type property and  the stability of the SL scheme. By using the half-relaxed limits technique introduced in \cite{BS91}, we show in Sect.~\ref{convergence_analysis} our main result, which is the convergence of solutions to the SL scheme towards the unique viscosity solution to \eqref{HJB_principal_equation}. The convergence is uniform in $[0,T]\times\ov{\OO}$ and holds under the same asymptotic condition between the space and time steps than in the case $\OO=\RR^N$. Next, in Sect.~\ref{numerical_results} we first illustrate the numerical convergence of the SL scheme in the case of a one-dimensional linear equation with  homogeneous Neumann boundary conditions. In this case the numerical results confirm  that the boundary condition in \eqref{HJB_principal_equation} is not satisfied at every $x\in \partial \OO$, but it is satisfied in the viscosity sense recalled in Sect.~\ref{prelim_results} below. In a second example, we consider a two dimensional degenerate second order nonlinear equation on a circular domain with non-homogeneous Neumann and oblique boundary conditions. In the last example, we consider a two-dimensional non-degenerate nonlinear equation on a non-smooth domain. Due to the lack of regularity of $\partial \OO$, our convergence result does not apply. However, the SL scheme can be successfully applied, which suggests that our theoretical findings could hold for more general domains. This extension as well as the corresponding study in the stationary framework remain as interesting subjects of future research. Finally, we provide in the Appendix of this work some theoretical results concerning oblique projections and the regularity of the distance to $\partial \OO$, which play a key  role in the definition of the scheme and in the proof of its main properties.\medskip

\section*{Acknowledgements}The first two authors would like to thank the Italian Ministry of Instruction, University and Research (MIUR) for supporting this research with funds coming from the PRIN Project $2017$ ($2017$KKJP$4$X entitled ``Innovative numerical methods for evolutionary partial differential equations and applications''). Xavier Dupuis thanks the support by the EIPHI Graduate School (contract
ANR-17-EURE-0002). Elisa Calzola, Elisabetta Carlini and Francisco J. Silva were partially supported by KAUST through the subaward agreement OSR-$2017$-CRG$6$-$3452$.$04$. 

\section{Preliminaries} \label{prelim_results}
As mentioned in the introduction, it will be simpler to describe our approximation scheme when \eqref{HJB_principal_equation} is written in backward form. This can be done by a simple change of the time variable and a possible modification of the time dependency of $H$. Let us set $\OO_T:=[0,T) \times \OO$ and $\ov{\OO}_{T}=[0,T]\times \ov{\OO}$. We consider the HJB equation 
\begin{equation}\label{hjb_continuous}\ba{rcl} -\partial_t u  + H\left(t,x, Du , D^{2}u \right)&=&0 \hspace{0.3cm} \mbox{in } \; \OO_T, \\[6pt]
L(t, x, D u)&=&0  \hspace{0.3cm} \mbox{on } \;  [0,T) \times \partial \OO,\\[6pt]
u(T,x)&=&\Psi(x)  \hspace{0.4cm} \mbox{in  } \ov{\OO},
\ea \tag{\rm HJB}
\end{equation}
where $H$ and $L$ are respectively given by \eqref{H_sup} and \eqref{B_sup}.

For notational convenience, throughout this article, we will write $\gamma_b(x)=\gamma(x,b)$ for all $x\in \partial \OO$ and $b\in B$.  Our standing assumptions for the data in {\rm(HJB)} are the following.
\begin{itemize}
\item[{\bf(H1)}] $\OO\subseteq \RR^{N}$ ($1\leq N\leq 3$) is a nonemtpy, bounded domain with   boundary $\partial \OO$  of class $C^3$.\vspace{0.2cm}
\item[\bf{(H2)}]  The functions $\sigma$,  $\mu$, $f$, $g$ and $\Psi$ are continuous. Moreover, for every $a\in A$, the functions $\sigma(\cdot, \cdot,a)$ and $\mu(\cdot, \cdot, a)$ are Lipschitz continuous, with Lipschitz constants independent of $a\in A$. \vspace{0.2cm}
\item[{\bf(H3)}] The function  $\gamma$ is of class $C^1$. We also assume that
$$(\forall \; (x,b) \in \partial \OO \times B) \quad |\gamma_b(x)|=1 \quad \mbox{and} \quad  \langle n(x), \gamma_b(x)\rangle > 0,
$$
where, for every $x\in \partial \OO$, we recall that $n(x)$ denotes the  unit outward normal vector  to $\ov{\OO}$ at $x$.
\end{itemize}

%
We now recall the notion of viscosity solution to ${\rm(HJB)}$ (see \cite{Barles1993}). We need first to introduce some notation.   Given a bounded function $z: \ov{\OO}_{T} \to \RR$, its upper semicontinuous (resp.  lower semicontinuous) envelope is defined by  
\be\label{relaxed_limits}
(\forall \, (t,x) \in\overline{\OO}_T) \quad z^{*}(t,x):= \underset{ \substack{(s,y)\in \ov{\OO}_T, \\    (s,y) \to (t,x)}}{\limsup} z(s,y) \hspace{0.3cm} \left( \mbox{resp. } \; z_{*}(t,x):=  \underset{ \substack{(s,y)\in \ov{\OO}_T, \\ (s,y) \to (t,x)}}\liminf z(s,y)\right). 
\ee
\begin{definition}\label{def:visc_sol} {\rm[Viscosity solution] }

{\rm(i)} An upper semicontinuous function $u_1: \ov{\OO}_{T}\to \RR$ is a viscosity subsolution to  {\rm(HJB)}  if for any $(t,x)\in \overline{\OO}_T$ and $\phi \in C^{2}(\overline{\OO}_T)$ such that $u_1-\phi$ has a local maximum at $(t,x)$, we have
\be \label{subsolution_inside} 
-\partial_t \phi(t,x) + H(t,x, D\phi(t,x), D^2\phi(t,x)) \leq 0,
\ee
if $(t,x)\in \OO_{T}$,
\be
\min\left\{ -\partial_t \phi(t,x) + H(t,x, D\phi(t,x), D^2\phi(t,x)), L(t,x, D\phi(t,x))  \right\} \leq 0, \label{subsolution_boundary}
\ee
if $(t,x)\in [0,T) \times \partial \OO$ and, 
\be \label{final_time_subsolution}  u_1(t,x)  \leq \Psi(x), \ee 
if  $(t,x) \in \{T\} \times \ov{\OO}$. 

{\rm(ii)} A lower semicontinuous function  $u_2: \ov{\OO}_{T}\to \RR$ is a viscosity supersolution to  {\rm(HJB)} if for any $(t,x)\in \overline{\OO}_T$ and $\phi \in C^{2}(\overline{\OO}_T)$ such that $u_2-\phi$ has a local minimum at $(t,x)$, we have
\be \label{supersolution_inside} 
-\partial_t \phi(t,x) + H(t,x, D\phi(t,x), D^2\phi(t,x)) \geq 0,
\ee
if $(t,x)\in \OO_{T}$,
\be
\max\left\{ -\partial_t \phi(t,x) + H(t,x, D\phi(t,x), D^2\phi(t,x)), L(t,x, D\phi(t,x))  \right\} \geq 0, \label{supersolution_boundary}
\ee
if $(t,x)\in [0,T) \times \partial \OO$ and, 
\be \label{final_time_supersolution} u_2(t,x)  \geq \Psi(x), \ee 
if  $(t,x) \in \{T\} \times \ov{\OO}$. 

{\rm(iii)} A bounded function $u: \ov{\OO}_T \to \RR$ is a viscosity solution to {\rm(HJB)} if $u^*$ and $u_{*}$, defined in \eqref{relaxed_limits}, are, respectively, sub- and supersolutions to {\rm(HJB)}. 
\end{definition}
\begin{remark}\label{final_time_weak_condition} As shown in \cite[Proposition 6]{MR2399437}, relation \eqref{final_time_subsolution} can be replaced by   
\be\label{subsolution_final_time_interior_space}
\min\left\{ -\partial_t \phi(t,x) + H(t,x, D\phi(t,x), D^2\phi(t,x)),  u_1(t,x)-\Psi(x) \right\} \leq 0,  
\ee
if $(t,x)\in \{T\} \times   \OO$,
and   
\be\label{subsolution_final_time_boundary_space}
\min\left\{ -\partial_t \phi(t,x) + H(t,x, D\phi(t,x), D^2\phi(t,x)), L(t,x, D\phi(t,x)), u_1(t,x)-\Psi(x) \right\} \leq 0,
\ee
if $(t,x)\in \{T\} \times \partial \OO$.  Similarly, condition \eqref{final_time_supersolution} can be replaced by
\be 
\max\left\{ -\partial_t \phi(t,x) + H(t,x, D\phi(t,x), D^2\phi(t,x)),  u_2(t,x)-\Psi(x) \right\} \geq 0,
\ee
if $(t,x)\in \{T\} \times   \OO$,  and   
\be
\max\left\{ -\partial_t \phi(t,x) + H(t,x, D\phi(t,x), D^2\phi(t,x)), L(t,x, D\phi(t,x)), u_2(t,x)-\Psi(x) \right\} \geq 0,
\ee
if $(t,x)\in \{T\} \times \partial \OO$. 
\end{remark}
The following well-posedness result  for ${\rm(HJB)}$ has been shown in \cite[Theorem II.1] {Barles1993} (see also \cite{MR2421330}).
\begin{theorem}
Assume {\bf(H1)-(H3)}. Then there exists a unique viscosity solution $u\in C(\ov{\OO})$ to {\rm(HJB)}. 
\end{theorem}
\begin{remark}\label{theoretical_remarks} {\rm(i)} {\rm[Comparison principle and uniqueness]} The existence of at most one solution to {\rm(HJB)} follows from the following comparison principle {\rm(}see \cite[Theorem II.1] {Barles1993} and  also \cite[Proposition 3.4]{MR2421330}{\rm)}. If $u_1: \ov{\OO}_T \to \RR$ is a bounded  viscosity subsolution to {\rm(HJB)} and $u_2:\ov{\OO}_T \to \RR$ is a bounded viscosity supersolution to {\rm(HJB)},  then 
$$u_1 \leq u_2 \quad \mbox{in} \quad \overline{\OO}_T.$$
{\rm(ii)} {\rm[Existence]} Once a comparison principle has been shown, the existence of a solution to {\rm(HJB)} follows usually from the existence of sub- and supersolutions to {\rm(HJB)} and Perron's method. In Sect.~\ref{convergence_analysis}, we construct sub- and supersolutions to {\rm(HJB)} as suitable limits of solutions to the approximation scheme that we present in the next section. Together with the comparison principle, this yields an alternative existence proof of solutions to {\rm(HJB)}.

An different and interesting technique to show the existence of a solution to {\rm(HJB)} is to consider a suitable stochastic optimal control problem, with controlled reflection of the state trajectory at the boundary $\partial \OO$, and to show that the associated value function is a viscosity solution to {\rm(HJB)}. This strategy has been followed in \cite{MR2421330}.
\smallskip\\ 
{\rm(iii)} {\rm[Continuity]} The continuity of the unique viscosity solution to {\rm(HJB)} follows directly from the comparison principle and the continuity properties required in the definition of sub- and supersolutions to {\rm(HJB)}. Notice that, as usual for parabolic problems with Neumann type boundary conditions, we do not require any  compatibility condition between $\Psi$ and the operator $L$ at the boundary $\partial \OO$. 
\end{remark}

\section{The fully discrete scheme}\label{fully_discrete_scheme}
We introduce in this section a fully discrete SL scheme that approximates the unique viscosity solution to $\rm(HJB)$.  Throughout this section, we assume that {\bf(H1)}-{\bf(H3)} are fulfilled.

\subsection{Discretization of the space domain $\OO$}{ 
Let us fix $\Delta x>0$ and consider a polyhedral domain $\OO_{\Delta x}\subseteq\RR^N$ such that 
\be\label{estim_o}
d(\OO,\OO_{\Delta x})=\inf\left\{|x-y| \, | \, x\in\OO, \, y\in \OO_{\Delta x}\right\}\leq C {(\Delta x)}^2,
\ee
for some $C>0$.  A construction of such a domain $\OO_{\Delta x}$ can be found in \cite[Section 3]{barrett_elliott_1986} for $N=2$ or $N=3$, which explain the dimension constraint in {\bf(H1)}. However, the results in the remainder of this article can be extended to $N>3$, provided that a numerical domain $\OO_{\Delta x}$ satisfying \eqref{estim_o} exists. Let $\T_{\Delta x}$ be a triangulation of $\OO_{\Delta x}$  consisting of  simplicial finite elements $\mathsf{T}$ with vertices in
$\G_{\Delta}=\{x_i \; | \; i=1, \hdots, N_{\Delta x}\}$ (for some $ N_{\Delta x}\in \NN$). We assume that $\Delta x$ is the mesh size, i.e. the maximum of the diameters of $\mathsf{T}\in  \T_{\Delta x}$, all the vertices on $\partial \OO_{\Delta x}$  belong to $\partial\OO$,  at most one face of each element $\mathsf{T}\in \T_{\Delta x}$, with vertices on $\partial\OO_{\Delta x}$, intersects   $\partial \OO_{\Delta x}$, and $\T_{\Delta x}$ satisfies the following regularity condition: there exists $\delta\in(0,1)$, independent of $\Dx$, such that  each $\mathsf{T} \in \T_{\Dx}$ is contained in a ball of radius $\Delta x/\delta$ and contains a ball of radius $\delta\Dx$. As in \cite{DeckelnickHinze2007}, we introduce an auxiliary exact triangulation $\widehat \T_{\Delta x}$  of $\overline{\OO}$ with vertices in $\G_{\Delta x}$. The boundary elements of $\widehat \T_{\Delta x}$ are allowed to be curved  and we have $$\ov \OO = \bigcup _{ \widehat{\mathsf{T}} \in \T_{\Delta x}} \widehat{\mathsf{T}}.$$

Denoting by $p_{\mathsf{T}}$ the projection on $\mathsf{T}\in \T_{\Delta x}$, the projection $p_{\Dx}:\ov \OO \rightarrow \ov \OO_{\Dx}\cap \ov \OO$ is defined by
$$
\begin{aligned}p_{\Dx}(x)=p_{\mathsf{T}}(x),\quad &\mbox{if $x\in \widehat{\mathsf{T}}\in \widehat \T_{\Delta x}$ }\\
&\mbox{and the element $\mathsf{T}\in \T_{\Delta x}$ has the same vertices than $\widehat{\mathsf{T}}$.}
\end{aligned}
$$
Set  $\I_{\Delta x}=\{1, \hdots, N_{\Delta x}\}$ and denote by $\{\psi_{i} \, | \, i\in \I_{\Delta x}\}$  the linear finite element $\mathbb{P}_1$ basis function on $\mathcal{T}_{\Delta x}$. More precisely, for each $i\in \I_{\Delta x}$, $\psi_i: \OO_{\Delta x}\to \RR$  is a continuous function, affine on each $\mathsf{T}\in \mathcal{T}_{\Delta x}$,  $0\leq \psi_{i} \leq 1$, $\psi_{i}(x_i)=1$, $\psi_{i}(x_j)=0$  for all $i$, $j\in \I_{\Delta x}$ with $i\neq j$, and $\sum_{i=1}^{N_{\Delta x}}\psi_{i}(x)=1$ for all $x\in \OO_{\Delta x}$.   For any $\phi: \G_{\Delta x}\to \RR$ its linear interpolation $I[\phi]$  on the mesh $\widehat \T_{\Dx}$ is defined by 
\be\label{interpolation} I \left[\phi\right](x):=\sum_{i=1}^{N_{\Delta x}}\psi_{i}({p_{\Dx}}(x)) \phi(x_i),\quad\text{for all $x\in\overline{\OO}$.}\ee
\begin{lemma}\label{Interp_estimate} Let $\phi \in C^2(\ov \OO)$ and denote by $\phi|_{\G_{\Delta x}}$ its restriction to $\G_{\Delta x}$. Then there exists a constant $C_\phi >0$, independent of  $\Delta x$,  such that  
\be\label{interp_estim}
\sup_{x\in\overline{\OO}}\; \big|\phi(x) - I \left[ \phi|_{\G_{\Delta x}}\right] (x)\big|\leq C_{\phi} (\Delta x)^2.
\ee
\end{lemma}
\begin{proof}
Let $x\in \ov{\OO}$ and let $\mathsf{T}\in \T_{\Delta x}$ and $\widehat{\mathsf{T}}\in \widehat \T_{\Delta x}$ be two elements having the same vertices and such that $x\in\widehat{\mathsf{T}}$. By the triangular inequality
$$
|\phi(x) - I \left[ \phi|_{\G_{\Delta x}}\right] (x)|  \leq  |\phi(x) -  \phi (p_{\mathsf{T}}  (x))| + |\phi (p_{\mathsf{T}} (x)) - I\left[\phi|_{\G_{\Delta x}}\right](x)|.
$$
Using that $\phi$ is Lipschitz, we deduce from \eqref{estim_o} the existence of $C_1>0$, independent of $\Delta x$ and $x\in \overline{\OO}$, such that $|\phi(x) -  \phi(p_{\mathsf{T}}  (x))|\leq C_1(\Delta x)^2$. In addition, by standard error estimates for $\PP_1$ interpolation (see for instance \cite{ciarlet}) and \eqref{interpolation}, there exists $C_2>0$, independent of $\Delta x$  and $x\in \overline{\OO}$, such that $|\phi (p_{\mathsf{T}} (x)) - I\left[\phi|_{\G_{\Delta x}}\right](x)|\leq C_2(\Delta x)^2$. Relation \eqref{interp_estim} follows from these two estimates. 
\end{proof}
\begin{figure}[h!]
\begin{center}
\includegraphics[width=0.5\textwidth]{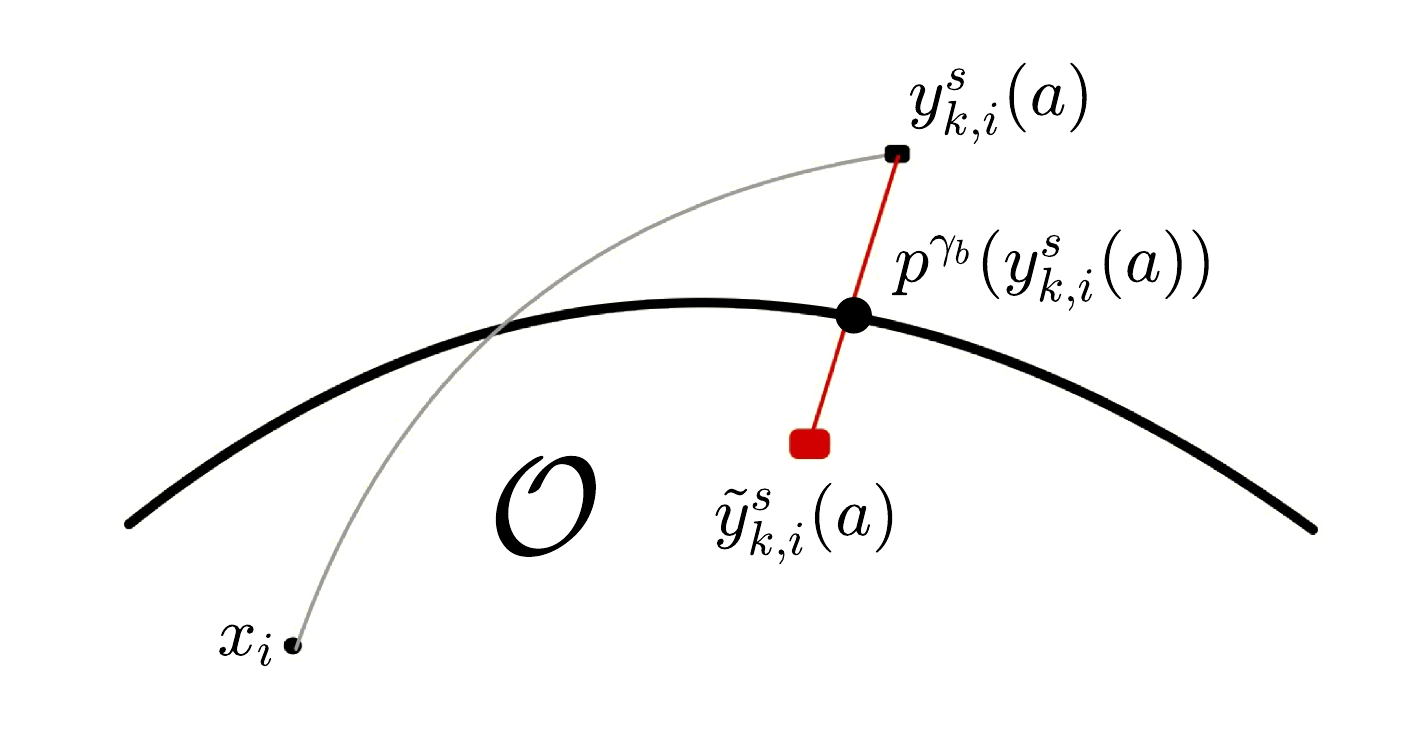}\end{center}
\caption{Reflection: reflected characteristic $\tilde y_{k,i}^{s}(a)$ (red square)  starting from $x_i$  (black circle), which exits from $\OO$ and arrives in  $y_{k,i}^{s}(a)$ (black square). The red segment represents the oblique direction $\gamma_b$ and  the black circle the  projected point $p^{\gamma_b}(y_{k,i}^{s}(a)$).} 
\label{fig:RC}\end{figure}
\subsection{A semi-Lagrangian scheme} Let $\Dt>0$, set $N_\Dt:=\floor {T/\Dt} $, $\I_{\Delta t}:=\{0, \dots, N_{\Delta t} \}$ and $\I_{\Delta t}^*:=\I_{\Delta t}\setminus\{N_{T}\}$. We define the time grid $\G_\Dt:=\{t_k \,  | \, t_k=k\Delta t, \, \, k\in \I_{\Delta t}\}$. 
Given  $(k,i) \in \I^*_{\Dt}\times  \I_{\Delta x}$, $a\in A$, and $\ell=1,\dots, N_{\sigma}$,  we define the discrete characteristics
\be\label{evaluated_characteristics}
y_{k,i}^{\pm,\ell}(a)=    x_i +  \Delta t  \mu \left(t_k,x_i, a \right) \pm \sqrt{N_{\sigma} \Delta t}\sigma^{\ell}(t_k,x_i,a).
\ee
Let $\I=\{+,-\} \times \{1, \hdots, N_{\sigma}\}$ and let $\bar c>0$ be a fixed constant. For any $\delta>0$ we set 
$$ (\partial \OO)_{\delta}:=\{x \in \RR^N \, | \, d(x,\partial \OO)<\delta\}.$$
By Proposition~\ref{oblique_projection_well_defined} in the Appendix, there exist  $R>0$ and two $C^1$ functions $(\partial \OO)_{R} \times B \ni (x,b)\mapsto p^{\gamma_b}(x) \in \partial \OO$ and $(\partial \OO)_{R} \times B \ni (x,b)\mapsto d^{\gamma_b}(x)\in\RR$, uniquely determined, such that   
\be\label{eq:def_proj-dist_prem}
x=p^{\gamma_b}(x) + d^{\gamma_b}(x)\gamma_b(p^{\gamma_b}(x)), \quad\text{for all }(x,b)\in  (\partial \OO)_{R} \times B. 
\ee
Therefore,   there exists $\overline{\Dt}>0$ such that for all $\Delta t \in [0,\overline{\Dt}]$,  $(k,i) \in \I^*_{\Dt}\times  \I_{\Delta x}$, $a\in A$, $b\in B$, and $s\in \I$, the {\it reflected characteristic} 
\be\label{definizione_R_n}
\tilde y_{k,i}^{s}(a,b) :=\begin{cases} y_{k,i}^{s}(a)  &\mbox{if }y_{k,i}^{s}(a) \in \overline \OO,\\
p^{\gamma_b}(y_{k,i}^{s}(a)) -\bar{c}\sqrt{\Dt}\gamma_b(p^{\gamma_b}(y_{k,i}^{s}(a))) 
 &\mbox{otherwise}  \end{cases} 
\ee
is well-defined.  In Figure~\ref{fig:RC} we illustrate  how the reflected characteristic is computed from the projection $p^{\gamma_b}(y_{k,i}^{s}(a))$ of $y_{k,i}^{s}(a)$ onto $\partial\OO$ parallel to $\gamma_b$.  Let us also set
\begin{align}
\tilde{d}^{s}_{k,i}(a,b) &:=\begin{cases} 0  &\mbox{if }y_{i,k}^{s}(a) \in \overline \OO,\\
d^{\gamma_b} (y_{k,i}^{s}(a)) +\bar c\sqrt{\Dt} 
 &\mbox{otherwise}, \end{cases}\label{definition_c} \\
\tilde{g}_{k,i}^{s}(a,b) &:=\begin{cases} 0  &\mbox{if }y_{k,i}^{s}(a) \in \overline \OO,\\
g\left(t_k, p^{\gamma_b} \left(y_{k,i}^{s}(a) \right),b \right) 
 &\mbox{otherwise}. \end{cases}\label{definizione_c_n_qw2qwq}
\end{align}
Notice that  if $y_{k,i}^{s}(a) \notin \overline \OO$, then \eqref{eq:def_proj-dist_prem},   \eqref{definizione_R_n}, and \eqref{definition_c} imply that 
\be\label{definizione_R_n2}\tilde y_{k,i}^{s}(a,b)= y_{k,i}^{s}(a)-\tilde{d}^{s}_{k,i}(a,b)\gamma_b\left(p^{\gamma_b}(y_{k,i}^{s}(a))\right). \ee
 For $(k,i)\in \I^*_{\Dt}\times  \I_{\Delta x}$ and $\Phi: \G_{\Delta x}\to\RR$, let us define ${\mathcal{S}}_{k,i}[\Phi]:A\times B\to \RR$ by
\be\label{eq:scheme_fullydiscrete_case_linear}
\mathcal{S}_{k,i}[\Phi](a,b):=
\frac{1}{2N_{\sigma}} \underset{s \in \I} \sum  \left[ I[\Phi](\tilde{y}^{s}_{k,i}(a,b)) +\tilde{d}^{s}_{k,i}(a,b)\tilde{g}_{k,i}^{s}(a,b)\right] +\Dt f(t_k,x_i,a) \ee
and set
\be\label{eq:scheme_fullydiscrete_case}
{S}_{k,i}[\Phi]    
 := \underset{\substack{a \in A,  \,  b \in B}}\inf \; \mathcal{S}_{k,i}[\Phi](a,b).
\ee

In the remainder of this work, we will consider the following fully discrete SL scheme to approximate the solution to ${\rm(HJB)}$.  
 \be\label{eq:def_fullydiscrete}
\begin{array}{rcl}
U_{k,i} &=& S_{k,i}\big[U_{k+1,(\cdot)}\big], \quad \text{for} \; (k,i)\in  \I^*_{\Dt}\times  \I_{\Delta x},\\[4pt]
U_{N_{\Dt},i}&=& \Psi(x_i), \quad \; \;  \; \; \;\,\text{for} \; i\in  \I_{\Delta x}. 
\end{array}   \tag{\rm HJB$_{\mbox{{\rm \tiny disc}}}$}
\ee

\subsection{Probabilistic interpretation of the scheme}\label{Sec:probabilistic} The fully-discrete SL to approximate the solution to \eqref{hjb_continuous} in the unbounded case, i.e. $\OO=\RR^d$, has a natural interpretation in terms of a discrete time, finite state, Markov control process (see e.g. \cite[Section 3]{CamFal95}). We show below that a similar interpretation holds for \eqref{eq:def_fullydiscrete}. The latter will play an important role in the stability analysis of \eqref{eq:def_fullydiscrete} presented in the next section. Given  $k\in \I_{\Delta t}^{*}$ and $a\in A$, $b\in B$, let us define the controlled transition law
\be\label{transition_probabilities_neumann_case}
p_{k, i,j}(a,b):= \frac{1}{2N_{\sigma}} \sum_{s\in \I} \beta_{j}(\tilde{y}_{k,i}^{s}(a,b)), \quad \mbox{for all } i, \, j \in \I_{\Delta x}. 
\ee
We say that $(\pi_{k})_{k\in \I_{\Delta t}^{*}}$ is a $N_{\Delta t}$-{\it policy} if for all $k\in \I_{\Delta t}^{*}$ we have $\pi_k: \G_{\Delta x} \to A\times B$. The set of $N_{\Delta t}$-policies is denoted by $\Pi_{N_{\Delta t}}$.  Let us fix $k\in \I^{*}_{\Delta t}$ and, for notational convenience, set $\mathfrak{X}_{k}=\G_{\Delta x}^{N_{\Delta t}-k+1}$.  Associated to  $x_i\in \G_{\Delta x}$ and $\pi \in \Pi_{N_{\Delta t}}$, there exists a probability measure $\PP^{k,x_{i},\pi}$ on $ 2^{\mathfrak{X}_k}$ (the powerset of $\mathfrak{X}_k$)  and a Markov chain $\{X_{m}  \, | \, m=k, \hdots, N_{\Delta t}\}$, with state space $\G_{\Delta x}$,  such that 
\be\label{transition_probabilities}
\PP^{k,x_{i},\pi}(X_k=x_i)=1  \quad \text{and} \quad \PP^{k,x_{i},\pi}(X_{m+1} =x_j \; | \; X_{m}=x_i )=p_{m,i,j}(\pi_{m}(x_i)),
\ee
for $m=k,\hdots, N_{\Delta t}-1$. Now, consider a family  $\{\xi_{k+1}, \hdots, \xi_{N_{\Delta t}}\}$ of $\RR^{N_{\sigma}}$-valued independent random variables, which are also independent of $\{X_m \, |  \, m=k, \hdots,N_{\Delta t}\}$, and  with common distribution given by
$$
\mathbb{P} ( \xi_{m} = \pm e_{\ell} ) = \frac{1}{2N_{\sigma}}, \quad \mbox{for } m= k+1,\hdots,N_{\Dt} \mbox{ and } \ell= 1, \dots, N_{\sigma},
$$ 
where $e_\ell$ denotes the $\ell$-th canonical vector of $\RR^{N_{\sigma}}$. By a slight abuse of notation (see  \eqref{evaluated_characteristics}), for $m=k,\hdots,N_{\Dt}-1$, $x_i\in \G_{\Delta x}$, and  $a\in A$, let us set  
\begin{equation}
y_{m}(x_i,a)=x_i+\Delta t\mu(t_{m},x_i,a)+ \sqrt{N_{\sigma}\Dt}\sigma(t_{m},x_i,a)\xi_{m+1}. \label{definition_stochastic_characteristic}
\end{equation}
For $m=k,\hdots,N_{\Dt}-1$, $x_i\in \G_{\Delta x}$,  $a\in A$, and $b\in B$,   define the random variable  
\begin{equation}\label{def:h_aux}
h(t_m,x_i,a,b) = \left\{\ba{ll}  0   &  \mbox{if } y_{m}(x_i,a)  \in \ov{\OO},\\[5pt]
\left(d^{\gamma_b}( y_{m}(x_i,a))+\bar{c}\sqrt{\Dt}\right)g(t_{m},p^{\gamma_b}(y_{m}(x_i,a)),b)  &  \mbox{otherwise.}
\ea \right.
\end{equation}
For all $i\in \I_{N_{\Delta x}}$ and  $\pi\in \Pi_{N_{\Dt}}$, let us define 
$$\ba{rcl}
J_{k,i}(\pi)&=& \EE_{\PP^{k,x_{i},\pi}} \left( \sum_{m=k}^{N_{\Dt}-1}\left[ \Delta t f (t_{m}, X_{m},\alpha_m)+ h(t_m, X_{m},\alpha_m,\beta_{m}\big)\right] + \Psi\big(X_{N_{\Dt}}\big)\right),\\[6pt]
J_{N_{\Delta t},i}(\pi)&=& \Psi(x_i),
\ea
$$
where, for notational convenience, we have denoted, respectively, by $\alpha_m$  and $\beta_{m}$ the first $N_{A}$ and the last $N_{B}$ coordinates of $\pi_{m}(X_m)$. Notice that, by construction and \eqref{eq:scheme_fullydiscrete_case_linear}, we have that
$$J_{k,i}(\pi)=\SS_{k,i}[J_{k+1, (\cdot) }(\pi)](\alpha_{k},\beta_{k}).$$
Moreover, setting 
$$
\ba{rcl}
\hat{U}_{k,i}&=&\inf_{\pi \in \Pi_{N_{\Dt}}} \; J_{k,i}(\pi),\\[2pt]
\hat{U}_{N_{\Delta t},i}&=& \Psi(x_i), 
\ea
$$
for all $i\in\G_{\Delta x}$,  the dynamic programming principle (see e.g. \cite[Theorem 12.1.5]{MR3587374}) implies that $\{\hat{U}_{k,i} \, | \, k\in \I_{\Dt}, \, i\in \I_{\Delta x}\}$ satisfies \eqref{eq:def_fullydiscrete}. Since the latter  has a unique solution, we deduce that $U_{k,i}=\hat{U}_{k,i}$ for all $k\in \I_{\Dt}$ and $i\in \I_{\Delta x}$.
\begin{remark} Scheme \eqref{eq:def_fullydiscrete} can thus be interpreted as a Markov chain discretization of an stochastic control problem with oblique reflection in the boundary {\rm(}see e.g. \cite{MR2421330}{\rm)}. 
\end{remark}

\section{Properties of the fully discrete scheme}\label{properties_fully_discrete_scheme}
In this section, we establish some basic properties of  \eqref{eq:def_fullydiscrete}.
\begin{proposition}\label{monotonicity} The following hold:\\
{\rm(i) (Monotonicity)}  For all  $U, V\colon \G_{\Dx}\to\RR$ with $U \leq V$, we have 
$$ \mathcal{S}_{k,i}[U] \leq \mathcal{S}_{k,i}[V],\quad\text{for } k\in \I^*_{\Dt}  \; \text{and } i\in   \I_{\Dx}.$$
{\rm(ii) (Commutation by constant)}  For any $c\in \RR$ and $U\colon\G_{\Dx}\to\RR$,  
$$ \mathcal{S}_{k,i}[U+c]=\mathcal{S}_{k,i}[U]+c,\quad\text{for } k\in \I^*_{\Dt}  \; \text{and } i\in   \I_{\Dx}. $$
\end{proposition}
\begin{proof} 
Both assertions follow directly from \eqref{eq:scheme_fullydiscrete_case_linear} and \eqref{eq:def_fullydiscrete}.
\end{proof}
We show in Proposition~\ref{consistency_scheme} below a consistency result for \eqref{eq:def_fullydiscrete}. 
For this purpose, let us set
\begin{eqnarray}
\H(t,x,p,M,a)  &=&  -\frac{1}{2} \Tr\left( \sigma(t,x,a)\sigma(t,x,a)^\top M \right) - \langle \mu (t,x,a),p \rangle - f(t,x,a)\label{H_a}, \\
\; &  \; & \text{for $(t,x,p,M,a)\in \ov{\OO}_{T} \times \RR^N \times \RR^{N\times N_{\sigma}} \times A$}, \nonumber\\
\L(t,x,p,b) &=&     \langle \gamma(x,b),p \rangle - g(t,x,b), \\
\; &  \; & \text{for } (t,x,p,b)\in [0,T]\times \partial \OO\times \RR^N \times B, \label{L_b}\nonumber
\end{eqnarray} 
 and for all $k\in\I^*_{\Delta t}$, $i\in\I_{\Delta x}$, $s\in\I$, $q\in\RR^N$, $a\in A$, and $b\in B$, define
\be\label{def:tilde_L}
\tilde{\L}_{k,i}^{s}(q,a,b):= \left\{\ba{ll} 0 & \mbox{if }  y^{s}_{k,i}(a) \in \ov \OO, \\[6pt]
								\L\left(t_k, p^{\gamma_b} (y_{k,i}^{s}(a)),q,b\right)  & \mbox{otherwise.}\ea\right.
\ee
\begin{proposition}[Consistency]\label{consistency_scheme} Let $\phi \in C^3 \left( \ov{\OO}\right)$ and denote by $\phi|_{\G_{\Dx}}$ its restriction to $\G_{\Dx}$. Then the following hold:
\begin{enumerate}[{\rm(i)}]
\item\label{consistency_linear} For all $k\in \I^{*}_{\Dt}$, $i\in\I_{\Dx}$, $a\in A$, and $b\in B$, we have 
$$\ba{ll}
\mathcal{S}_{k,i}[\phi|_{\G_{\Dx}}](a,b) -\phi(x_i)
=&-\Dt \H (t_k,x_i,D\phi(x_i), D^2\phi(x_i),a)\\[8pt]
&-\frac{1}{2N_{\sigma}} \underset{s\in \I} \sum
 \tilde d^{s}_{k,i}(a,b) \left( \tilde{\L}_{k,i}^{s}(D \phi(x_i),a,b) - \sqrt{\Dt} K^{s}_{k,i}(a,b) \right) \\[10pt]
  \; & + \, O \left( \Dt \sqrt{\Dt} + (\Delta x)^2\right),
\ea$$ 
where the set of constants $\{K_{k,i}^{s}(a,b) \, | \, k\in\I^*_{\Delta t}, \, i \in\I_{\Delta x}, \, s\in \I, \, a\in A, \, b\in B \}$ is bounded, independently of $(\Delta t, \Delta x)$.
\item\label{consistency_nonlinear} For all $k\in \I^{*}_{\Dt}$ and $i\in\I_{\Dx}$, we have 
$$\ba{ll}
{S}_{k,i}[\phi|_{\G_{\Dx}}] -\phi(x_i)=&-\underset{\substack{a \in A,  \;  b \in B}}\sup\bigg\{\Delta t  \H (t_k,x_i, D\phi(x_i), D^2\phi(x_i),a)  \\[8pt]
& \hspace{0.5cm}+ \frac{1}{2N_{\sigma}} \underset{s\in \I} \sum
 \tilde d^{s}_{k,i}(a,b) \left( \tilde{\L}_{k,i}^{s}(D \phi(x_i),a,b) - \sqrt{\Dt} K^{s}_{k,i}(a,b) \right)\bigg\}\\[12pt]
&  \hspace{0.5cm}+ O \left( \Dt \sqrt{\Dt} + (\Delta x)^2\right).
\ea$$ 
\end{enumerate}

\end{proposition}
\begin{proof}   In what follows, we denote by $C>0$  a generic constant, which is independent of $k$, $i$, $s$ $a$, $b$, $\Delta t$ and $\Delta x$.  Since assertion {\rm(ii)} follows directly from ${\rm(i)}$, we only show the latter. 

For every $s\in\I$, \eqref{evaluated_characteristics} and \eqref{definition_c} imply that $0\leq \tilde{d}^{s}_{k,i}(a,b)\leq C\sqrt{\Dt}$. Thus, by \eqref{evaluated_characteristics}, \eqref{definizione_R_n2}, and a second order Taylor expansion of $\phi$ around $x_i$, for every $\ell=1,\hdots,N_{\sigma}$,  we have
$$\ba{lll}
\phi \left( \tilde{y}^{\pm,\ell}_{k,i}(a,b)  \right)\\
\hspace{0.5cm}=\phi \left( x_i \right) + \Dt  \langle D \phi(x_i),\mu (t_k,x_i,a)\rangle + \frac{N_{\sigma} \Dt }{2} \langle D^2\phi(x_i) \sigma^{\ell}(t_k,x_i,a), \sigma^{\ell}(t_k,x_i,a) \rangle  \\[8pt] 
\hspace{0.5cm}\pm \sqrt{N_{\sigma} \Dt } \langle D \phi(x_i), \sigma^{\ell} (t_k,x_i,a) \rangle  - \tilde{d}^{\pm,\ell}_{k,i}(a,b)\left\langle D \phi(x_i), \tilde{\gamma}_{k,i}^{\pm,\ell}(a,b) \right\rangle\\[8pt]
\hspace{0.5cm} + \frac{\left(  \tilde{d}^{\pm,\ell}_{k,i}(a,b) \right)^2}{2} \left\langle D^2 \phi(x_i)\tilde{\gamma}_{k,i}^{\pm,\ell}(a,b), \tilde{\gamma}_{k,i}^{\pm,\ell}(a,b) \right\rangle \\[8pt]
\hspace{0.5cm} \mp \sqrt{N_{\sigma}\Dt} \tilde{d}^{\pm,\ell}_{k,i}(a,b)\left\langle D^2 \phi(x_i)\tilde{\gamma}_{k,i}^{\pm,\ell}(a,b), \sigma ^{\ell}(t_k,x_i,a) \right\rangle + O \left( \Dt \sqrt{\Dt} \right), \\[8pt]
\ea$$
where, for every $s\in\I$,
$$
\tilde{\gamma}_{k,i}^{s}(a,b):= \begin{cases}  0  & \text{if } y_{k,i}^{s}(a)\in \ov{\OO}, \\[6pt]
									  \gamma_b \left(p^{\gamma_b} (y_{k,i}^{s}(a))\right)  & \text{otherwise}.
\end{cases}
$$

This implies that
\be\label{eq:escono}\ba{l}
\half \phi \left( \tilde{y}^{+,\ell}_{k,i}(a,b) \right)+\half \phi \left( \tilde{y}^{-,\ell}_{k,i}(a,b)  \right)  \; \\[8pt]
\hspace{1cm}=   \phi(x_i) + \Delta t \left\langle D\phi(x_i),\mu(t_k,x_i,a)\right\rangle  + \frac{N_{\sigma} \Delta t}{2}  \left\langle D^2\phi(x_i) \sigma^{\ell}(t_k,x_i,a),\sigma^{\ell}(t_k,x_i,a)\right\rangle \\[8pt]
 \hspace{1.2cm}- \tilde{d}^{+,\ell}_{k,i}(a,b) \left( \left\la D\phi(x_i),\tilde{\gamma}_{k,i}^{+,\ell}(a,b) \right\rangle - \sqrt{\Dt} K^{+,\ell}_{k,i}(a,b)\right) \\[8pt]
 \hspace{1.2cm}  - \tilde{d}^{-,\ell}_{k,i}(a,b) \left( \left \langle D\phi(x_i),\tilde{\gamma}_{k,i}^{-,\ell}(a,b)\right\ra - \sqrt{\Dt}K^{-,\ell}_{k,i}(a,b)\right) +  O \left( \Dt \sqrt{\Dt} \right),
\ea
\ee
where

$$ 
\begin{aligned}K_{k,i}^{\pm,\ell}(a,b) :=&  
\frac{\tilde{d}^{\pm,\ell}_{k,i}(a,b)}{2\sqrt{\Dt}}\langle D^2 \phi(x_i)\tilde{\gamma}_{k,i}^{\pm,\ell}(a,b), \tilde{\gamma}_{k,i}^{\pm,\ell}(a,b) \rangle \\ 
&\mp \sqrt{N_{\sigma}}   \langle D^2 \phi(x_i)\tilde{\gamma}_{k,i}^{\pm,\ell}(a,b), \sigma ^{\ell}(t_k,x_i,a) \rangle.
\end{aligned}
$$

Multiplying \eqref{eq:escono} by $1/N_{\sigma}$ and taking the sum over $s\in \I$,  we obtain
$$
\ba{lll}
\frac{1}{2N_{\sigma}}\underset{s\in\I}{\sum} \phi ( \tilde{y}^{s}_{k,i}(a,b)) \\
\hspace{1cm}=\phi(x) +\Delta t \langle D\phi(x_i),\mu(t_k,x_i,a)\rangle  + \frac{\Dt}{2} \Tr\left( \sigma(t_k,x_i,a)\sigma(t_k,x_i,a)^T D^2\phi(x_i) \right) \\[8pt]
\hspace{1.2cm}- \frac{1}{2N_{\sigma}} \underset{s\in\I}\sum
\tilde d^{s}_{k,i}(a,b) \left( \left\la D\phi(x_i),\tilde{\gamma}_{k,i}^{s}(a,b) \right\rangle - \sqrt{\Dt} K^{s}_{k,i}(a,b) \right) \\[8pt]
\hspace{1.2cm}+ O \left( \Dt \sqrt{\Dt} \right),
\ea
$$
which, by Lemma~\ref{Interp_estimate}, yields
$$
\ba{ll}
\frac{1}{2N_{\sigma}}\underset{s\in\I}{\sum} I[\phi|_{\G_{\Dx}}] ( \tilde{y}^{s}_{k,i}(a,b)) \\
\hspace{1cm} =\phi(x) +\Delta t \langle D\phi(x_i),\mu(t_k,x_i,a)\rangle  + \frac{\Dt}{2} \Tr\left( \sigma(t_k,x_i,a)\sigma(t_k,x_i,a)^T D^2\phi(x_i) \right) \\[8pt]
\hspace{1.2cm} - \frac{1}{2N_{\sigma}} \underset{s\in\I}\sum
\tilde d^{s}_{k,i}(a,b) \left( \left\la D\phi(x_i),\tilde{\gamma}_{k,i}^{s}(a,b) \right\rangle - \sqrt{\Dt} K^{s}_{k,i}(a,b) \right) \\[8pt]
\hspace{1.2cm}+ O \left( \Dt \sqrt{\Dt} +  (\Dx)^2 \right). 
\ea
$$
The result follows from the previous expression, \eqref{eq:scheme_fullydiscrete_case_linear}, \eqref{H_a} and \eqref{def:tilde_L}.
\end{proof}

For $k\in \I_{N_{\Delta t}}^*$ and $a\in A$, let us define
\be\label{gamma_k}
(\forall \,k \in \I^*_{N_{\Dt}}, \forall\, a \in A )\quad\Gamma_k (a):= \{ x_{i} \in \G_{\Delta x} \; | \; \exists \; s\in \I, \; y_{k,i}^s(a) \notin \ov{\OO}\},
\ee
and recall from Sect.~\ref{Sec:probabilistic} that given    $x_i\in \G_{\Delta x}$ and  a policy $\pi \in \Pi_{N_{\Delta t}}$, the Markov chain $\{X_{m}  \, | \, m=k, \hdots, N_{\Delta t}\}$ is defined by the transition probabilities  \eqref{transition_probabilities}. As in Sect.~\ref{Sec:probabilistic}, we denote by $\alpha_{m}$ and $\beta_{m}$ ($m=k, \hdots, N_{\Delta t}-1$),  respectively, the first $N_{A}$ and the last $N_{B}$ coordinates of $\pi_{m}(X_m)$. Finally, given $D\subset \RR^d$, we denote by $\II_{D}$ the indicator function of $D$, i.e. $\II_{D}(x)=1$, if $x\in D$, and  $\II_{D}(x)=0$, otherwise. 

The following technical result will be useful to establish the stability of \eqref{eq:def_fullydiscrete}.
\begin{lemma}\label{lemma:milstein_stab_d}  The following holds:  
\be\label{sojourn_time}
\sup_{k\in \I^{*}_{\Delta t}, \; i\in \I^{*}_{\Delta x}, \pi \in \Pi_{N_{\Delta t}}}  \EE_{\PP^{k,x_{i},\pi}}\left( \sum_{m=k}^{N_{T}-1} \mathbb{I}_{\Gamma_m \left(\alpha_{m}\right)}\big(X_{m}\big)   \right) \leq \frac{C}{\sqrt{\Dt}},\ee
where $C>0$ is independent of $(\Dt, \Dx)$ as long as $\Delta t$ is small enough and $(\Delta x )^2 /\Dt$ is bounded.
\end{lemma}
\begin{proof} The argument of the proof is inspired from  \cite[Lemma 1]{Milstein96}.
Let $\eps>0$,  set
$$
D_\eps =  \{ x \in \overline \OO \; | \; d(x, \partial \OO) > \eps \}, \quad  \partial D_\eps =  \{ x \in \overline \OO \; | \; d(x, \partial \OO) =\eps \},$$ 
$$ L_{\eps}=  \{ x \in \overline \OO \; | \; d(x, \partial \OO)  \le  \eps \},
$$
and define $\ov{\OO} \ni  x\mapsto  w_\eps(x)=d^2\left(x,D_{\eps}\right)\in \RR$.
By Lemma \ref{lemma:dist1}{\rm(v)} in the Appendix, there exists $\eta>0$ such that $w_{\eta} \in  C^3(\overline{\OO}\setminus  \partial D_{\eta})$ with bounded third order derivatives on the connected components of $\overline{\OO}\setminus  \partial D_{\eta}$. Let us fix this $\eta$ and, for notational convenience, let us write $w=w_\eta$.  Let $M>0$ and, for any $k \in \I_{\Delta t}$, define  
\be
\ov{\OO} \ni x\mapsto    
W_k(x)= \begin{cases}
M(T - t_k) + w(x) \quad&\mbox{if } k \in  \I^*_{\Delta t},\\
0  &\mbox{if } k=N_{\Delta t} 
\end{cases}\in \RR.
\ee
By \eqref{eq:scheme_fullydiscrete_case_linear}, with $f\equiv0$ and $g\equiv0$, for all $a\in A$ and $b\in B$,  we have 
\begin{eqnarray} 
\mathcal{S}_{k,i}[W_{k+1}|_{\G_{\Dx}}](a,b) - W_k(x_i) &=& -M\Dt +\mathcal{S}_{k,i}[w|_{\G_{\Dx}}](a,b)  - w (x_i), \label{eq:pw-w1d} \\
\; & = &  -M\Dt +\frac{1}{2N_{\sigma}}\sum_{s\in \I}I[w](\tilde{y}_{k,i}^s(a,b))- w (x_i).\label{eq:pw-w2d}
\end{eqnarray}
Moreover, assumption {\bf(H2)} implies the existence of $\ov{C}>0$  such that
\be\label{displacement_bound}\sup\left\{  |y_{k,i}^{s}(a)-x_i| \; \bigg| \; k \in \I^{*}_{\Dt}, \, i\in \I_{\Dx}, \, a\in A, \, s \in \I\right\} \leq \ov{C}\sqrt{\Dt}.
\ee
Now, let us fix  $k \in \I_{\Delta t}^{*}$, $i \in \I_{\Delta x}$, $a\in A$, and $b\in B$. We have the following cases.

{\rm(i)} \underline{$x_i \notin \Gamma_k (a)$ and $d(x_i,\partial D_{\eta})\geq \ov{C}\sqrt{\Dt}$}. The first condition implies that $y^{s}_{k,i}(a)\in \overline\OO$,  for any $s\in \I$, and, hence, \eqref{definizione_R_n} yields $\tilde{y}_{k,i}^s(a,b) =y^{s}_{k,i}(a)$. The condition $d(x_i,\partial D_{\eta})\geq \ov{C}\sqrt{\Dt}$, \eqref{displacement_bound}, and  standard error estimates for $\PP_1$ interpolation (see for instance \cite{ciarlet}), imply that
$$
I[w](\tilde{y}_{k,i}^s(a,b))= w(\tilde{y}_{k,i}^s(a,b)) + O((\Delta x)^2)=w(y^{s}_{k,i}(a)) + O((\Delta x)^2).
$$
Since, by second order Talyor expansion, $ \frac{1}{2N_{\sigma}}\sum_{s\in \I}w(y_{k,i}^s(a))-w(x_i)=O(\Delta t)$,   \eqref{eq:pw-w2d} yields  
\be\label{eq:stimadentro}
\mathcal{S}_{k,i}[W_{k+1}|_{\G_{\Dx}}]( a, b)- W_k(x_i) =  -M \Dt + O \left( \Dt + (\Delta x )^2 \right).
\ee

{\rm(ii)} \underline{$x_i \notin \Gamma_k (a)$ and $d(x_i,\partial D_{\eta})< \ov{C}\sqrt{\Dt}$}. Condition $d(x_i,\partial D_{\eta})< \ov{C}\sqrt{\Dt}$  and  \eqref{displacement_bound} imply that $w(x_i)=O(\Delta t)$ and, for any $s\in \I$, $d^2(y^{s}_{k,i}(a),\partial D_{\eta})=O(\Delta t)$. Since the cardinality of  $\J:= \{ j \in \I_{\Delta x} \, | \,  \psi_j(y^{s}_{k,i}(a)) >0\}$ is independent of $\Dx$  and, for all $j \in \J$, $| y^{s}_{k,i}(a)-x_j| =O(\Dx)$, we deduce that 
$$\ba{rcl}
I[w](y^{s}_{k,i}(a))&=&\sum_{j\in \J}\psi_j(y^{s}_{k,i}(a))w(x_j)\\[5pt]
\; & \leq&\sum_{j\in \J}\psi_j(y^{s}_{k,i}(a))d^2(x_j,\partial D_{\eta})\\[5pt]
\; &=&\sum_{j\in \J}\psi_j(y^{s}_{k,i}(a))d^2(y^{s}_{k,i}(a),\partial D_{\eta})+O((\Dx)^2)\\[5pt]
\; & =& O( \Delta t + (\Dx)^2).
\end{array}$$
Thus, since $\tilde{y}_{k,i}^s(a,b) =y^{s}_{k,i}(a)$, \eqref{eq:pw-w2d} implies that \eqref{eq:stimadentro} still holds.  

{\rm(iii)} \underline{$ x_i \in   \Gamma_k (a)$.} Let $0<\delta<\eta$. Since $\mu$ and $\sigma$ are bounded, there exists $\ov{\Delta t}>0$, independent of $k$, $i$ and $a$, such that 
\be\label{gamma_k_inc}
\Gamma_k (a) \subseteq L_{\delta}\subset L_{\eta},\ee 
if $\Delta t \leq \ov{\Delta t}$. By 
\eqref{eq:pw-w1d} and  Proposition~\ref{consistency_scheme}{\rm (i)},  with $f\equiv0$ and $g\equiv0$,  we have  
\be\label{eq:pw-w4d}
\ba{ll}
\mathcal{S}_{k,i}[W_{k+1}|_{\G_{\Dx}}]( a, b) - W_k(x_i)=\\[8pt]
\hspace{1cm} -M\Dt  - \frac{1}{2N_{\sigma}} \sum_{s\in \I} \tilde{d}^{s}_{k,i}(a,b) \left\langle Dw (x_i),\gamma_b\left(p^{\gamma_b} \left(y^{s}_{k,i}(a)\right) \right) \right\rangle\\[8pt]
\hspace{1cm}  +O \left( \Dt + (\Delta x )^2 \right).
\ea
\ee
By Lemma~\ref{lemma:dist1}{\rm(v)} in the Appendix,  for any $x \in L_\eta$, we have $d\left(x, \partial D_\eta \right) = \eta - d(x, \partial \OO)$. Thus,  Lemma~\ref{lemma:dist1}{\rm(ii)} implies that  $Dd\left(x, \partial D_\eta \right)=n(p_{\partial \OO}(x))$, and  hence  
\be\label{derivata_w_prova}
Dw(x_i) = 2d\left(x_i, \partial D_\eta \right)Dd\left(x_i, \partial D_\eta \right) = 2d\left(x_i, \partial D_\eta \right)n(p_{\partial \OO}(x)).
\ee
On the other hand, in view  of \cite[Proposition 1.1{\rm(v)}]{gobet2001}, there exists $C>0$ such that $|d^{\gamma_b} (x_i)|\leq C d(x_i,\partial \OO)$.
Thus,
$$
\begin{aligned}
\lvert p^{\gamma_b} (x_i) - p_{\partial \OO}( x_i) \rvert   & \leq  \lvert p^{\gamma_b} (x_i) - x_i \rvert + \lvert x_i - p_{\partial \OO}(x_i) \rvert=|d^{\gamma_b} (x_i)|  +  d(x_i,\partial \OO) \\
&\leq (C+1) d(x_i,\partial \OO).
\end{aligned}
$$
Since $x_i\in \Gamma_k(a)$, we have $d(x_i,\partial \OO)=O(\sqrt{\Delta t})$ and hence $\lvert p^{\gamma_b} (x_i) - p_{\partial \OO}( x_i) \rvert=O(\sqrt{\Dt})$.  Proposition \ref{oblique_projection_well_defined}  implies that $\gamma_b$ and $p^{\gamma_b}$ are Lipschitz and hence, for any $s\in \I$,
\be\label{eq:gammay}
\gamma_b\left(p^{\gamma_b} \left(y^{s}_{k,i}(a)\right) \right) = \gamma_b \left(p^{\gamma_b} (x_i)\right) + O\left(\sqrt{\Dt}\right) = \gamma_b \left(p_{\partial \OO}(x_i)\right) + O \left(\sqrt{\Dt}\right).
\ee
Since, for all $s\in \I$, $ \tilde{d}^{s}_{k,i}(a,b) = O (\sqrt{\Dt})$, from \eqref{eq:pw-w4d}-\eqref{eq:gammay} we obtain 
\be\label{eq:pw-w4_2d}
\ba{l}
\mathcal{S}_{k,i}[W_{k+1}|_{\G_{\Dx}}]( a,b) - W_k(x_i)  =\\ [6pt]
\hspace{2cm} -M\Dt  -\frac{1}{N_{\sigma}} \sum_{s\in \I}d\left(x_i, \partial D_\eta \right) \tilde{d}^{s}_{k,i}(a,b)\big\langle 
n(p_{\partial \OO}(x_i)),\gamma_b \left( p_{\partial \OO}(x_i)\right) \big\rangle  \\ [6pt]
\hspace{2cm}+O\left(   \Dt +(\Delta x )^2 \right).
\ea
\ee
Since $x_i\in\Gamma_k(a)$, there exists $\widetilde \I_{k,i} \subset \I \neq \emptyset$ such that $\tilde{d}^{s}_{k,i}(a,b) > 0$, for  any $s\in \widetilde \I_{k,i} $. In addition,   \eqref{gamma_k_inc} implies  that $d\left(x_i,\partial D_\eta \right)\geq \eta - \delta>0$.
Thus,  assumption {\bf (H3)} implies that 
$$
\mathcal{S}_{k,i}[W_{k+1}|_{\G_{\Dx}}](a, b)- W_k(x_i)  \leq -M \Dt -\frac{\nu(\eta - \delta)}{N_{\sigma}} \sum_{s \in \widetilde \I_{k,i}}  \tilde{d}^{s}_{k,i}(a,b)+ O\left(\Dt +(\Delta x )^2 \right),
$$
and hence \eqref{definition_c} yields the existence of $C>0$, independent of $k\in \I_{\Dt}^*$, $i\in \I_{\Dx}$, $a\in A$, and $b\in B$,   such that 
\be\label{eq:pw-w4_4d}
\mathcal{S}_{k,i}[W_{k+1}|_{\G_{\Dx}}](a, b)- W_k(x_i)  \leq -M \Dt - C \sqrt{\Dt}+ O\left(   \Dt +(\Delta x )^2 \right).
\ee
As long as  $(\Delta x )^2 /\Dt$ is bounded, we have that $ O\left(\Dt +(\Delta x )^2 \right)=O(\Delta t)$. Thus, from cases {\rm(i)}-{\rm(iii)} we can choose $M$ large enough such that 
\be\label{eq:pw-w4_4daasa}
\mathcal{S}_{k,i}[W_{k+1}|_{\G_{\Dx}}](a, b)- W_k(x_i)  \leq  - C \sqrt{\Dt}\II_{\Gamma_{k}(a)}(x_i). 
\ee
Now, set $q_{k}(x_i,a,b) =W_k(x_i) -\mathcal{S}_{k,i}[W_{k+1}|_{\G_{\Dx}}](a, b)$. Then the probabilistic interpretation of the operator $\mathcal{S}_{k,i}$ (see Sect.~\ref{Sec:probabilistic}) implies that, for any policy $\pi\in \Pi_{N_{\Delta t}}$,
$$
W_k(x_i)=\EE_{\PP^{k,x_{i},\pi}}\left( \sum_{m=k}^{N_{T}-1}  q_{m}\big(X_{m},\alpha_{m},\beta_m\big) + w\big(X_{N_{T}}\big)  \right).
$$
Since \eqref{eq:pw-w4_4daasa} implies that $q_{k}(x_i,a,b)\geq C \sqrt{\Dt}\II_{\Gamma_{k}(a)}(x_i)$  for $k\in \I_{\Dt}^{*}$, $i\in \I_{\Dx}$, $a\in A$ and $b\in B$, we deduce that for any policy $\pi\in \Pi_{N_{\Delta t}}$ we have 
$$
\ba{rcl}
\EE_{\PP^{k,x_{i},\pi}}\left( \sum_{m=k}^{N_{T}-1} \mathbb{I}_{\Gamma_m \left(\alpha_{m}\right)}\big(X_{m}\big)   \right) &\leq& \ds \frac{1}{C\sqrt{\Delta t}}\EE_{\PP^{k,x_{i},\pi}}\left( \sum_{m=k}^{N_{T}-1}  q_{m}\big(X_{m},\alpha_{m},\beta_m\big)   \right)\\[12pt]
\; &=& \ds \frac{W_{k}(x_i)-\EE_{\PP^{k,x_{i},\pi}}\big(w\big(X_{N_{T}}\big) \big)}{C \sqrt{\Dt}}.
\ea
$$
Finally, using that $W_{k}$ and $w$ are bounded, \eqref{sojourn_time} follows. 
\end{proof}
\begin{proposition}\label{stability}{\rm(Stability)}
The fully discrete scheme \eqref{eq:def_fullydiscrete} is stable, i.e. there exists $C>0$  such that 
\be\label{eq:stability_result_constant}\underset{k\in \I^{*}_{\Dt}, \, i\in \I_{\Delta x} }\max \lvert U_{k,i} \rvert \leq C,
\ee
where $C$ is independent of $(\Dt, \Dx)$ as long as $\Delta t$ is small enough and $(\Delta x )^2 /\Dt$ is bounded. 
\end{proposition}
\begin{proof}
 Let us fix $k \in \I^*_{\Dt}$ and $i\in \I_{\Delta x}$. Then the probabilistic interpretation of the scheme in Sect.~\ref{Sec:probabilistic} and the definition of $h$ in \eqref{def:h_aux} imply the existence of a constant $C>0$ such that 

$$
\ba{rcl}
|U_{k,i}|&\leq& \ds \sup_{\pi \in \Pi_{N_{\Delta t}}}  \EE_{\PP^{k,x_{i},\pi}}\Big(  \sum_{m=k}^{N_{\Dt}-1} \left[ \Dt\big|f (t_m,X_m,\alpha_m)\big| \right.\ \\[13pt]
& & \hspace{0.5cm}  +\left.\big|h(t_m, X_{m},\alpha_m,\beta_{m}\big)\big|\right]+ \big|\Psi\big(X_{N_{\Dt}}\big)\big| \Big)\\[6pt]
 &\leq&  \|\Psi\|_{\infty}+T\|f\|_{\infty}+ C\sqrt{\Dt}\|g\|_{\infty} \ds  \sup_{\pi \in \Pi_{N_{\Delta t}}}  \EE_{\PP^{k,x_{i},\pi}}\left(\sum_{m=k}^{N_{\Dt}-1} \II_{\Gamma_m ( \alpha_{m})}\left(X_m\right)\right).
\end{array}
$$

Thus, \eqref{eq:stability_result_constant} follows from Lemma \ref{lemma:milstein_stab_d}.
\end{proof}
\section{Convergence analysis}\label{convergence_analysis}
In this section we provide the main result of this article which is the convergence of solutions to \eqref{eq:def_fullydiscrete} to the unique viscosity solution of \eqref{hjb_continuous}. The proof is based on the half-relaxed limits technique introduced in \cite{BS91} and the properties of solutions to \eqref{eq:def_fullydiscrete} investigated in Sect.~ \ref{properties_fully_discrete_scheme}.

Let $\Delta t>0 $, let $\Delta x>0$ and let $(U_k)_{k=0}^{N_{\Dt}}$ be the solution to \eqref{eq:def_fullydiscrete} associated to the discretization parameters $\Delta t$ and $\Delta x$. Let us define  an extension  of $(U_k)_{k=0}^{N_{\Dt}}$ to  $\ov{\OO}_T$ by 
\be\label{definition_extension_of_the_discrete_solutions}
(\forall \; (t,x) \in \ov{\OO}_T) \quad u_{\Dt,\Delta x} (t,x):=I[U_{\floor{t/\Dt}}](x),
\ee
where we recall that the interpolation operator $I[\cdot]$ is defined in \eqref{interpolation}. Now, let $(\Dt_n,\Delta x_n)_{n\in \NN} \subseteq (0,+\infty)^2$ be  such that $\lim_{n\to \infty}(\Dt_n,\Delta x_n)=(0,0)$ and the sequence $(\Delta x_n/ \Delta t_n)_{n\in \NN}$ is bounded. For every $(t,x)\in \ov{\OO}_{T}$, let us define 
\be\label{definition_relaxed_limits}
\begin{split}
\uo (t,x):= \underset{ \substack{ n\to \infty\\ \ov{\OO}_T \ni (s_n,y_n)\to (t,x)}}\limsup \; u_{\Dt_n,\Delta x_n}(s_n,y_n), \\
\uu(t,x):= \underset{ \substack{ n\to \infty\\ \ov{\OO}_T \ni (s_n,y_n)\to (t,x)}}\liminf\; u_{\Dt_n,\Delta x_n}(s_n,y_n).
\end{split}
\ee
From Proposition~\ref{stability} we deduce that $\uo \colon \ov{\OO}_{T}\to \RR$ and $\uu\colon \ov{\OO}_{T}\to \RR$ are well-defined and bounded. Moreover, from \cite[Chapter V, Lemma 1.5]{BardiCapuzzo96}, we have that   $\uo$ and $\uu$ are, respectively, upper and lower semicontinuous functions.
\begin{proposition} \label{Prop:Viscositysub-super} Assume that $(\Delta x_{n})^{2}/\Dt_n \to 0$, as $n \to \infty$.   Then $\uo$ and  $\uu$ are,   respectively,  viscosity sub- and supersolutions to \eqref{hjb_continuous}. 
\end{proposition} 

\begin{proof} We only show that $\uo$ is a viscosity subsolution to  \eqref{hjb_continuous}, the proof that $\uu$ is a viscosity supersolution being similar. Let  $(\bar{t},\bar{x})  \in  \ov{\OO}_T$ and $\phi \in C^{\infty}( \ov{\OO}_T)$ be such that $\ov u(\bar{t},\bar{x})=\phi(\bar{t},\bar{x})$ and $\ov u- \phi$ has a maximum at $(\bar{t},\bar{x})$. Then by \cite[Chapter V, Lemma 1.6]{BardiCapuzzo96}  there exists a subsequence of $(u_{\Dt_n,\Delta x_n} )_{n\in \NN}$, which for simplicity is still labeled by $n\in \NN$,  and a sequence $(s_n,y_n)_{n\in \NN}\subseteq \ov{\OO}_T$ such that $(u_{\Dt_n,\Delta x_n} )_{n\in \NN}$ is uniformly bounded,  $u_{\Dt_n,\Delta x_n}-\phi$ has a local maximum at $(s_n,y_n)$, and, as $n\to \infty$,  $(s_n,y_n)\to (\bar{t},\bar{x}) $ and $u_{\Dt_n,\Delta x_n}(s_n,y_n) \to \ov u(\bar{t},\bar{x})$. Moreover, by modifying the test function $\phi$, we can assume that $u_{\Dt_n,\Delta x_n}-\phi$ has a global maximum at $(s_n,y_n)$, i.e. setting  $\xi_{n}:= u_{\Dt_n,\Delta x_n}(s_n,y_n)- \phi(s_n,y_n)$, 
we have 
\be\label{eqeqeqeqae}
(\forall \; (t,x) \in \ov{\OO}_T) \quad  u_{\Dt_n,\Delta x_n}(t,x) \leq   \phi(t,x)  +\xi_{n}, \hspace{0.5cm} {\textrm{with}}\; \xi_n \to 0.
\ee
We  distinguish now   the following cases.

{\bf(i)} \underline{$(\bar{t},\bar{x})  \in [0,T)\times \OO$}. In this case,  for all $n$   large enough, by \eqref{estim_o},  we have $y_n \in \OO_{\Delta x_n}$. Let $k : \mathbb{N} \to \I_{\Delta t_n}^{*}$ be such that $s_n \in [t_{k(n)}, t_{k(n)+1})$.  As $n\to \infty$, we have
 $t_{k(n)} \to \bar t$ and, from \eqref{definition_extension_of_the_discrete_solutions} and  \eqref{eqeqeqeqae}, with $t=t_{k(n)+1}$,  we have
\be\label{inequality_dimostra_1}
(\forall \; x\in \overline{\OO}) \quad I [U_{k(n)+1}](x) \leq   \phi(t_{k(n)+1},x)  +\xi_{n}.
\ee
From Proposition \ref{monotonicity}, we obtain
\be\label{inequality_dimostra_2}
(\forall \; i \in \I_{\Dx})  \quad S_{k_n, i}[U_{k(n)+1}] \leq  S_{k_n, i}[\Phi_{k(n)+1} ]  +\xi_{n},  
\ee
where, for all $k\in \I_{\Dt}$, we have denoted $\Phi_k:= \phi(t_k,\cdot)|_{\mathcal{G}_{\Delta x_n}}$. In particular, by \eqref{eq:def_fullydiscrete}  we get
\be\label{inequality_dimostra_3}
(\forall \; i \in \I_{\Dx}) \quad U_{k(n),i} \leq   S_{k_n, i}[\Phi_{k(n)+1} ]  +\xi_{n}.
\ee
The monotonicity of the interpolation operator \eqref{interpolation} yields  
\be\label{inequality_dimostra_4}
\left(\forall \; x \in \ov{\OO}\right) \quad u_{\Dt_n,\Delta x_n}(s_n,x)\leq   \sum_{i\in\I_{\Dx_n}}\psi_{i}\big(p_{{\Delta x_n}}(x)\big)  S_{k_n,i}[\Phi_{k(n)+1} ] +\xi_{n},
\ee
and hence, by taking $x=y_n$  and using the definition of $\xi_{n}$, we obtain
\be\label{eq:dis_mon_fd}
\phi(s_n,y_n)\leq   \sum_{i\in\I_{\Dx_n}} \psi_{i}(y_n) S_{k_n,i}[\Phi_{k(n)+1} ]. 
\ee
Since $(\bar{t},\bar{x})  \in [0,T)\times \OO$ and $A, B$ are compacts, if $n$ large enough, for all $a\in A, b\in B$ and for all $s\in \I$  we have  $\tilde d^{s}_{k_n,i}(a,b)=0$ for all $i \in \I_{\Dx}$ such that $\psi_{i}(y_n)>0$. 
Using Proposition~\ref{consistency_scheme}\eqref{consistency_nonlinear} and inequality \eqref{eq:dis_mon_fd}, we get 
$$\ba{ll}
\phi(s_n,y_n)&\leq \ds \underset{i\in\I_{\Dx_n}} \sum\psi_{i}(y_n) \left[\phi(t_{k(n)+1},x_i) -\right.\\
& \hspace{1.2cm}\left. \Dt_n
\underset{a \in A }\sup \; \H\left( t_{k(n)},x_i,  D\phi(t_{k(n)+1},x_i), D^2\phi(t_{k(n)+1},x_i) ,a\right)\right] \\[20pt]
& \hspace{0.4cm}+ O\left(\Dt_n\sqrt{\Dt_n}+(\Delta x_n)^2\right).
\ea $$
Then following the same arguments than those in \cite[Theorem 3.1]{CFF10} (see also \cite[Theorem 4.22]{falconeferretilibro}) we conclude that 
\be\label{subsolution_inequality_vera}-\partial_t\phi(\bar t,\bar x)+H(\bar t,\bar x,   D\phi(\bar{t}, \bar{x}), D^2\phi(\bar{t},\bar{x}))\leq 0,
\ee
and, hence, \eqref{subsolution_inside}  holds.

{\bf(ii)} \underline{$(\bar{t},\bar{x})\in  [0,T) \times \partial \OO$}. If 
$$L(\bar{t},\bar{x}, D \phi(\bar{t},\bar{x})) \leq 0 \quad \text{or} \quad -\partial_t\phi(\bar t,\bar x)+H(\bar t,\bar x, D\phi(\bar t,\bar x), D^2\phi(\bar t,\bar x)) \leq  0,$$ 
holds, then \eqref{subsolution_boundary} holds. Thus, let us suppose that
\be\label{aseaweqeq}
L(\bar{t},\bar{x},D \phi(\bar{t},\bar{x}))  > 0 \quad \mbox{and} \quad
-\partial_t\phi(\bar t,\bar x)+H(\bar t,\bar x, D\phi(\bar t,\bar x), D^2\phi(\bar t,\bar x)) > 0.
\ee
Letting $k: \NN \to \{0,\hdots, N_{T}-1\}$ as in {\bf(i)}, we have $t_{k(n)} \to \bar{t}$, \eqref{inequality_dimostra_4} holds true, and hence, 
\be\label{eq:dis_mon_fd_nuova}
\phi(s_n,y_n)\leq   \sum_{i\in \I_{\Dx_n}} \psi_{i}\big( p_{\Dx_n}(y_n)\big) S_{k_n,i}[\Phi_{k(n)+1} ]. 
\ee
On the one hand, from  Proposition~\ref{consistency_scheme}\eqref{consistency_nonlinear}  we get  
$$\ba{cl} 
0 &\leq    \ds \underset{i\in\I_{\Dx_n}} \sum\psi_{i}( p_{\Delta x_n}(y_n)) \bigg( \Dt_n\partial_{t} \phi (t_{k(n)},x_i) \\
&\hspace{0.4cm} -\underset{\substack{a \in A, \\  b \in B}}\sup \bigg\{\Dt_n\H (t_{k(n)},x_i, D\phi(t_{k(n)+1},x_i), D^2\phi(t_{k(n)+1},x_i) ,a)   \\ 
&\hspace{0.4cm}+\frac{1}{2N_{\sigma}}  \ds \sum_{s\in \I} \tilde{d}^{s}_{k,i}(a,b)\left(\tilde{\L}^{s}_{k(n),i}(D \phi(t_{k(n)+1},x_i),a,b) - \sqrt{\Dt}_n K^{s}_{k(n),i}(a,b)  \right)  \bigg\}\bigg) \\
&\hspace{0.4cm} +    O\left(\Dt_n\sqrt{\Dt}_n+(\Delta x_n)^2\right) 
\ea$$
and hence, for all $a\in A$ and $b\in B$,  we have 
\be\label{eq:fix_a_discr}
\ba{cl} 
& \ds \underset{i\in\I_{\Dx_n}}  \sum \psi_{i}\big( p_{\Delta x_n}(y_n)\big) \bigg\{ -\Dt_n\partial_t\phi(t_{k(n)},x_i) \\[8pt]
&\hspace{1cm} + \Dt_n\H (t_{k(n)},x_i, D\phi (t_{k(n)+1},x_i), D^2\phi (t_{k(n)+1},x_i) ,a)     \\[8pt]
&\hspace{1cm} + \frac{1}{2N_{\sigma}}  \ds  \sum_{s\in \I}\tilde{d}^{s}_{k,i}(a,b)\left(\tilde{\L}^{s}_{k(n),i)}(D \phi(t_{k(n)+1},x_i),a,b) -  \sqrt{\Dt}_n K^{s}_{k(n),i}(a,b) \right) \bigg\}   \\[8pt]
 &\hspace{1cm} +    O\left(\Dt_n\sqrt{\Dt}_n+(\Delta x_n)^2\right)\leq 0.
\ea
\ee
On the other hand, since $A$ is compact, there exists $\bar{a}\in A$ such that
$$
 H(\bar t,\bar x, D\phi(\bar t,\bar x), D^2\phi(\bar t,\bar x))=\H (\bar t,\bar x, D\phi(\bar t,\bar x), D^2\phi(\bar t,\bar x),\bar{a}) 
$$
and 

\be\label{213nnn2222}\ba{cl}&\ds \underset{i\in\I_{\Dx_n}} \sum\psi_{i}\big( p_{\Delta x_n}(y_n)\big) \left(-  \partial_t\phi(t_{k(n)},x_i) \right.\\[6pt]
 &  \hspace{2cm} +\left. \H (t_{k(n)},x_i, D\phi(t_{k(n)+1},x_i), D^2\phi(t_{k(n)+1},x_i) ,\bar{a})  \right)\\[6pt]
  & \hspace{2cm} \to -\partial_t\phi(\bar t,\bar x)+H(\bar t,\bar x, D\phi(\bar t,\bar x), D^2\phi(\bar t,\bar x)), \quad  \text{as $n\to \infty$.}
\ea\ee

Let us set $\tilde{d}_n^* =\max \left\{ \tilde{d}^{s}_{k_n,i}(\bar{a}) \; \big| \; s\in\I,  \; {i} \in \I_{\Delta x_n}\right\}$ and take $a=\bar{a}$ and an arbitrary $b\in B$  in \eqref{eq:fix_a_discr}. If there exists a subsequence, still labelled by $n$,  such that $\tilde{d}_{n}^* =0$, then dividing \eqref{eq:fix_a_discr} by $\Dt_n$,  and letting $n\to \infty$, \eqref{213nnn2222} yields 
$$  -\partial_t\phi(\bar t,\bar x)+H(\bar t,\bar x, D\phi(\bar t,\bar x), D^2\phi(\bar t,\bar x))\leq 0,$$
which contradicts \eqref{aseaweqeq}.  Otherwise,  by \eqref{definition_c}, for all $n \in \NN$, large enough,  we have  $\tilde{d}_n^* \geq \bar{c} \sqrt{\Dt_n}$. Notice that the second relation in \eqref{aseaweqeq} and \eqref{213nnn2222} imply that, for $n\in \NN$ large enough,  

\be\label{awqwq1eeeee} 
\ba{cl}
&0 <   \ds \underset{i\in\I_{{\Dx}_n}}\sum \psi_{i}\big( p_{\Delta x_n}(y_n)\big) \left(-  \partial_t\phi(t_{k(n)},x_i)\right. \\[6pt]
  & \hspace{2cm}\left. +     \H (t_{k(n)},x_i, D\phi(t_{k(n)+1},x_i), D^2\phi(t_{k(n)+1},x_i) ,\bar{a})  \right).
\ea
\ee

Therefore, inequality \eqref{eq:fix_a_discr} with $a=\bar{a}$ implies that for all $b\in B$  
\be\label{aeqennnfnfq}
\ba{l} 
    \ds \underset{i\in\I_{{\Dx}_n}} \sum \psi_{i}\big( p_{{\Delta x_n}}(y_n)\big) \bigg\{   \sum_{s\in\I}\tilde{d}^{s}_{k_n,i}(\bar{a},b)\left(\tilde \L^s_{k_n,i} (
   D \phi(t_{k(n)+1},x_i ),\bar{a},b)\right. \\
    \hspace{1cm}-\left.\sqrt{\Dt}_n K^{s}_{k(n),i}(\bar{a},b) \right)\bigg\} +  O\left(\Dt_n\sqrt{\Dt}_n+(\Delta x_n)^2\right)<0.
\ea
\ee
Since the set $\I=\{+,-\} \times \{1, \hdots, d\}$ is finite, there exist  $\hat{s} \in \I$, $\{q^{s} \; | \; s \in   \I \setminus \{\hat{s}\}\} \subseteq [0,1]$, and $i(n)\in \I_{\Delta x_n}$ such that, up to some subsequence,  $\tilde{d}_n^*  =\tilde{d}^{\hat{s}}_{k(n),i(n)}(\bar{a})$ and, for all $s\in \I\setminus \{\hat{s}\}$,  $\tilde{d}^{s}_{k(n),i(n)}( \bar{a})/ \tilde{d}_n^{*}  \to q^{s}$.   Recall that $\tilde{d}_n^*\geq \bar{c} \sqrt{\Dt_n}$ and  $(\Delta x_{n})^{2}/\Dt_n \to 0$ as $n\to\infty$. Dividing \eqref{aeqennnfnfq} by $\tilde{d}_n^*$ and taking the limit $n\to \infty$  yields 
$$
(\forall \; b\in B) \quad \left(\sum_{s\in \I \setminus \{\hat{s}\}} q^s +1\right) \L(\bar{t},\bar{x},D \phi(\bar{t},\bar{x}), b) \leq 0$$
$$ \quad \mbox{and  hence} \quad (\forall \; b\in B)\quad \L(\bar{t},\bar{x},D \phi(\bar{t},\bar{x}), b)\leq 0.
$$
Thus,  $L(\bar{t},\bar{x},D \phi(\bar{t},\bar{x}))\leq 0$, which contradicts \eqref{aseaweqeq}. 

%
%
{\bf(iii)} \underline{$(\bar{t}, \bar{x}) \in \{T\} \times \overline{\OO}$.}
Let us first assume that $(\bar{t}, \bar{x}) \in \{T\} \times {\OO}$. Thus, for $n\in \NN$ large enough, we have $y_n\in\OO$. By taking a subsequence, if necessary, it suffices to consider the cases  $s_n \in [0,T)$, for all $n\in \NN$, and $s_n=T$, for all $n\in \NN$. In the first case, proceeding as in  {\bf(i)}, we get 
\be\label{hamilonian_inequality_final_time}  -\partial_t\phi(\bar t,\bar x)+H(\bar t,\bar x, D\phi(\bar t,\bar x), D^2\phi(\bar t,\bar x)) \leq  0.
\ee
In the second case, \eqref{definition_extension_of_the_discrete_solutions} implies that $u_{\Dt_n,\Delta x_n} (s_n,y_n)=I[\Psi|_{\G_{\Dx}}](y_n)$ and hence letting $n\to \infty$ we get
\be\label{equality_final_condition}\ov u(\bar{t},\bar{x})= \Psi(\bar{x}).
\ee

Now, assume that $(\bar{t}, \bar{x}) \in \{T\} \times \partial \OO$. As before, it suffices to consider the cases  $s_n \in [0,T)$, for all $n\in \NN$, and $s_n=T$ for all $n\in \NN$. If  $ s_n \in  [0,T)$, then, proceeding as in {\bf(ii)}, we get 
\be\label{hamilonian_inequality_final_time_x_in_the_boundary}
L(\bar{t},\bar{x}, D \phi(\bar{t},\bar{x})) \leq 0 \quad \text{or} \quad -\partial_t\phi(\bar t,\bar x)+H(\bar t,\bar x, D\phi(\bar t,\bar x), D^2\phi(\bar t,\bar x)) \leq  0.
\ee
Finally, if $s_n=T$, for all $n\in\NN$,  we have $u_{\Dt_n,\Delta x_n} (s_n,y_n)=I[\Psi|_{\G_{\Dx}}](y_n)$ and hence \eqref{equality_final_condition} holds. 

Altogether, \eqref{hamilonian_inequality_final_time} and \eqref{equality_final_condition} imply that \eqref{subsolution_final_time_interior_space} holds if $(\bar{t}, \bar{x}) \in \{T\} \times {\OO}$, and \eqref{hamilonian_inequality_final_time_x_in_the_boundary} and \eqref{equality_final_condition} imply that \eqref{subsolution_final_time_boundary_space}  holds if $(\bar{t}, \bar{x}) \in \{T\} \times \partial \OO$. 

Thus, from cases {\bf(i)-(iii)} and Remark~\ref{final_time_weak_condition} we obtain that $\ov{u}$ is a subsolution to \eqref{hjb_continuous}.
\end{proof}

\begin{theorem}\label{Th:conv}
Assume {\normalfont \bf{(H1)-(H3)}} and that $(\Delta x_{n})^{2}/\Dt_n \to 0$, as $n \to \infty$.  Then 
$$u_{\Dt_n, \Delta x_n} \to u \quad \text{uniformly in } \ov{\OO}_{T},$$ 
where $u$ is the unique continuous viscosity solution to \eqref{hjb_continuous}.
\end{theorem}
\begin{proof} By \eqref{definition_relaxed_limits} we have $\uu \leq \uo$ in $\ov{\OO}_{T}$ and, by Proposition~\ref{Prop:Viscositysub-super} and the comparison principle for sub- and super solutions to  \eqref{hjb_continuous} (see Remark~\ref{theoretical_remarks}{\rm(i)}), we obtain that  $\uu \geq \uo$ in $\ov{\OO}_{T}$. Thus, $u=\uu = \uo$ and the result follows from \cite[Chapter V, Lemma 1.9]{BardiCapuzzo96}.
\end{proof}
\section{Numerical results}\label{numerical_results}

In this section, we present some numerical experiments in order to show the performance of the scheme.
We consider first a one-dimensional linear parabolic equation, with homogeneous Neumann boundary conditions, and  both the first and second order cases. In the former, the boundary conditions are not satisfied in the pointwise sense at every point in the boundary,  but they hold in the viscosity sense (see Definition~\ref{def:visc_sol}). The second example deals with a degenerate second order nonlinear equation on a smooth two-dimensional domain. We consider both non-homogeneous Neumann and oblique boundary conditions. In the last example, we approximate the solution to a non-degenerate second order nonlinear equation with mixed Dirichlet and homogeneous Neumann boundary conditions on a non-smooth domain. Because of the presence of Dirichlet boundary conditions and corners, the scheme has to be modified and the convergence result in Sect.~\ref{properties_fully_discrete_scheme} does not apply. However, the scheme can be successfully applied to solve numerically the problem.

The problems in the first two tests  have known analytic solutions. This will allow to compute the errors of solutions to the scheme and to perform a numerical convergence analysis. In the examples dealing with two-dimensional domains, we have considered unstructured triangular meshes, constructed  with the Matlab2019 function \verb+initmesh+.

%

In the simulations we have chosen  time and space steps satisfying $\Delta t= \Delta x$ or $\Delta t= \Delta x /2$, which are in agreement with the assumption in Theorem \ref{Th:conv}.

\subsection{One-dimensional linear problem}\label{test1} Let $\eps> 0$, set $\lambda^{\pm}_{\varepsilon} =  (1 \pm \sqrt{1+4\varepsilon})/2\eps$, and define 
$$
\ba{l} 
\; f(t,x) = \ds  \frac{3-t}{2} \left(1 + \frac{ e^{\lambda^+_{\varepsilon} x} \left(e^{\lambda^-_{\varepsilon}} - 1\right)}{e^{\lambda^+_{\varepsilon}} -  e^{\lambda^-_{\varepsilon}}} \left(1 -\varepsilon \lambda^+_{\varepsilon} \right) + \frac{ e^{\lambda^-_{\varepsilon}x} \left(1 -  e^{\lambda^+_{\varepsilon}}\right)}{e^{\lambda^+_{\varepsilon}} -  e^{\lambda^-_{\varepsilon}}} \left(1 - \varepsilon \lambda^-_{\varepsilon} \right) \right) \\[14pt]
\hspace{3cm} \ds +\frac{1}{2}\left(   x + \frac{e^{\lambda^+_{\varepsilon} x} \left(e^{\lambda^-_{\varepsilon}} - 1\right)}{e^{\lambda^+_{\varepsilon}} -  e^{\lambda^-_{\varepsilon}}} + \frac{ e^{\lambda^-_{\varepsilon}x} \left(1 -  e^{\lambda^+_{\varepsilon}}\right)}{e^{\lambda^+_{\varepsilon}} -  e^{\lambda^-_{\varepsilon}}} \right), \\[14pt]\ds
u_\varepsilon(t,x) =   \frac{3-t}{2}\left( x + \frac{ e^{\lambda^-_{\varepsilon}} - 1}{\lambda^+_{\varepsilon}\left( e^{\lambda^+_{\varepsilon}} -  e^{\lambda^-_{\varepsilon}} \right)} e^{\lambda^+_{\varepsilon} x} + \frac{1 -  e^{\lambda^+_{\varepsilon}}}{\lambda^-_{\varepsilon}\left( e^{\lambda^+_{\varepsilon}} -  e^{\lambda^-_{\varepsilon}} \right)} e^{\lambda^-_{\varepsilon} x}\right),
\ea
$$
for $(t,x)\in [0,1]^2$. Then $u_\eps$ is the unique classical solution to 
\be\label{eq:test1}
\ba{l}
\, -\partial_{t} u -\eps \partial^{2}_{x} u  + \partial_{x} u     =  f  \quad  \mbox{in }  [0,1) \times (0,1),\\[4pt]
  \partial_{x}u(\cdot,0)=\partial_x u(\cdot,1) =  0 \quad \mbox{in $[0,1)$},\\[4pt]
\hspace{2.1cm}u(1,\cdot)= u_{\eps}(1,\cdot) \quad \mbox{in $[0,1]$}.
\ea
\ee


Similarly to  \cite[Example 7.3]{CrandallIshiiLions92}, we have 
$$
u_\varepsilon (t,x) \underset{\eps\to 0}{\longrightarrow} u_0(t,x):=  \frac{3-t}{2} \left( x + e^{-x} \right), \quad {\mbox{uniformly on }}[0,1]^2
$$
and $u_0$ is the unique viscosity solution to 
\be\label{eq:test1_null_eps}
\ba{l}
\hspace{1.22cm} -\partial_{t} u   + \partial_{x} u     =  f  \quad  \mbox{in }  [0,1) \times (0,1),\\[4pt]
 \partial_{x}u(\cdot,0)=\partial_x u(\cdot,1) =  0 \quad \mbox{in $[0,1)$},\\[4pt]
\hspace{2.1cm}u(1,\cdot)= u_{0}(1,\cdot) \quad \mbox{in $[0,1]$}.
\ea
\ee

Notice that for $t\in [0,1]$ we have $-\partial_{t} u(t,1)   + \partial_{x} u(t,1)-f(t,1)\leq 0$ and $\partial_x u(t,1)>0$. Thus, at $(t,1)$ the boundary condition is satisfied in the viscosity sense but not in the pointwise sense. 


Using \eqref{eq:def_fullydiscrete}, we approximate $u_\eps$ for $\varepsilon =0.05$, $\varepsilon =0.03$,  and $\varepsilon =0$. 
For these choices, we plot in Figure \ref{fig:test1_nuovo} respectively the  approximations of $u_{\eps}(1,\cdot)$  and $u_{\eps}(0,\cdot)$, computed with the steps sizes $\Delta x = 3.125 \cdot 10^{-3}$ and $\Dt = \Delta x/2$.

We show  in Tables 1 and 2 the errors 
$$ E_\infty=\ds \underset {i\in \I_{\Dx}}\max |U_{0,i}-u(0,x_i)|,\quad 
         E_{1}=\Delta x\ds \sum_{i\in \I_{\Dx}} |U_{0,i}-u(0,x_i)|,\quad 
$$
and the corresponding convergence rates $p_{\infty}$ and $p_1$, for  $\varepsilon = 0.05$ and  $\varepsilon = 0$, respectively. In all cases, an order of convergence close to $1$ is obtained.
%
%

In the simulations, we have chosen $\bar{c}:=0.025+ \sigma/2$, where $\sigma = \sqrt{2\varepsilon}$ is the diffusion parameter.  With this choice, the larger the value of $\sigma$, the more the characteristics are reflected further into $\OO$.
\begin{table}[h!]
  \centering
  \caption{Errors and convergence rates for  problem \eqref{eq:test1} with $\varepsilon = 0.05$.}
\label{table:eps_nuovo}
\begin{tabular}{lllllllll} 
\hline\noalign{\smallskip}
& \multicolumn{4}{c|}{$\Dt = \Delta x$} & \multicolumn{4}{c}{$\Dt = \Delta x/2$} \\
\noalign{\smallskip}\hline\noalign{\smallskip}
$\Delta x$& $E_{\infty}$ &  $E_1$ & $p_{\infty}$ &  $p_1$ & $E_{\infty}$ &  $E_1$ & $p_{\infty}$ &  $p_1$    \\
\hline
$5.00\cdot 10^{-2}$ &  $3.99\cdot 10^{-2}$ & $2.57\cdot 10^{-2}$  & -  &  - & $2.16\cdot 10^{-2}$ & $2.03\cdot 10^{-2}$  & -  &  -  \\
\hline 
$2.50\cdot 10^{-2}$ &  $2.25\cdot 10^{-2}$ & $1.06\cdot 10^{-2}$  & 0.83& 1.28  & $1.26\cdot 10^{-2}$ & $6.22\cdot 10^{-3}$  & 0.78 &  1.71 \\
\hline
$1.25\cdot 10^{-2}$ & $1.17\cdot 10^{-2}$ & $6.13\cdot 10^{-3}$  & 0.94  &  0.79 & $5.87\cdot 10^{-3}$ & $5.64\cdot 10^{-3}$  & 1.10&  0.14  \\
\hline 
$6.25\cdot 10^{-3}$ &  $5.38\cdot 10^{-3}$ & $2.49\cdot 10^{-3}$  & 1.12 & 1.30 & $3.17\cdot 10^{-3}$ & $2.95\cdot 10^{-3}$  & 0.89 &  0.93 \\
\hline 
$3.125 \cdot 10^{-3}$ &  $2.15\cdot 10^{-3}$ & $1.77\cdot 10^{-3}$  & 1.32  & 0.49 & $1.62\cdot 10^{-3}$ & $1.50\cdot 10^{-3}$  & 0.97 & 0.98\\
\hline 
 \end{tabular}
\end{table}
\begin{table}[h!]
\caption{Errors and convergence rates for  problem \eqref{eq:test1} with $\varepsilon = 0$.}
\label{table:eps0_nuovo}
  \centering
\begin{tabular}{lllllllll} 
\hline\noalign{\smallskip}
& \multicolumn{4}{c|}{$\Dt = \Delta x$} & \multicolumn{4}{c}{$\Dt = \Delta x/2$} \\
\noalign{\smallskip}\hline\noalign{\smallskip}
$\Delta x$ & $E_{\infty}$ &  $E_1$ & $p_{\infty}$ &  $p_1$ & $E_{\infty}$ &  $E_1$ & $p_{\infty}$ &  $p_1$    \\
\hline
$5.00\cdot 10^{-2}$ &  $2.83\cdot 10^{-2}$ & $1.95\cdot 10^{-2}$  & -  &  - & $2.26\cdot 10^{-2}$ & $1.86\cdot 10^{-2}$  & -  &  -  \\
\hline 
$2.50\cdot 10^{-2}$ &  $1.42\cdot 10^{-2}$ & $1.01\cdot 10^{-2}$  & 0.99 & 0.95  & $1.15\cdot 10^{-2}$ & $9.97\cdot 10^{-3}$  & 0.97 &  0.90 \\
\hline
$1.25\cdot 10^{-2}$ & $7.08\cdot 10^{-3}$ & $5.39\cdot 10^{-3}$  & 1.00  &  0.91 & $5.88\cdot 10^{-3}$ & $5.42\cdot 10^{-3}$  & 0.97 &  0.88  \\
\hline 
$6.25\cdot 10^{-3}$ &  $3.54\cdot 10^{-3}$ & $2.91\cdot 10^{-3}$  & 1.00  & 0.89 & $3.04\cdot 10^{-3}$ & $2.97\cdot 10^{-3}$  & 0.95 &  0.87  \\
\hline 
$3.125 \cdot 10^{-3}$ &  $1.77\cdot 10^{-3}$ & $1.59\cdot 10^{-3}$  & 1.00  & 0.87 & $1.68\cdot 10^{-3}$ & $1.63\cdot 10^{-3}$  & 0.86 &  0.87  \\
\hline 
 \end{tabular}

\end{table}
\begin{figure}
\centering
       \includegraphics[width=0.4\textwidth]{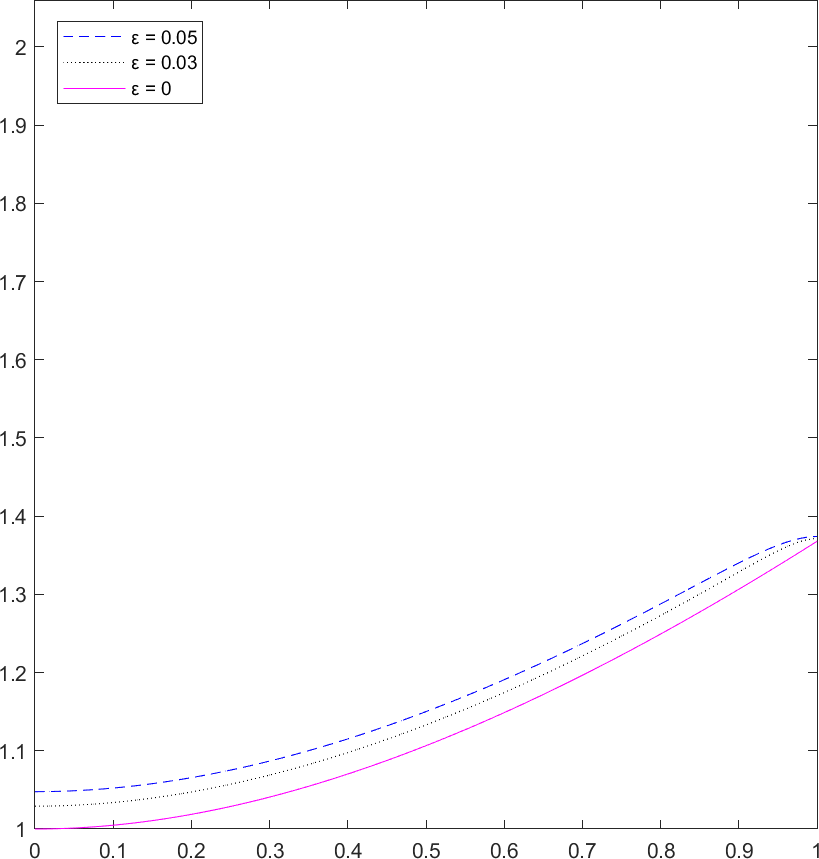}
         \centering
         \includegraphics[width=0.4\textwidth]{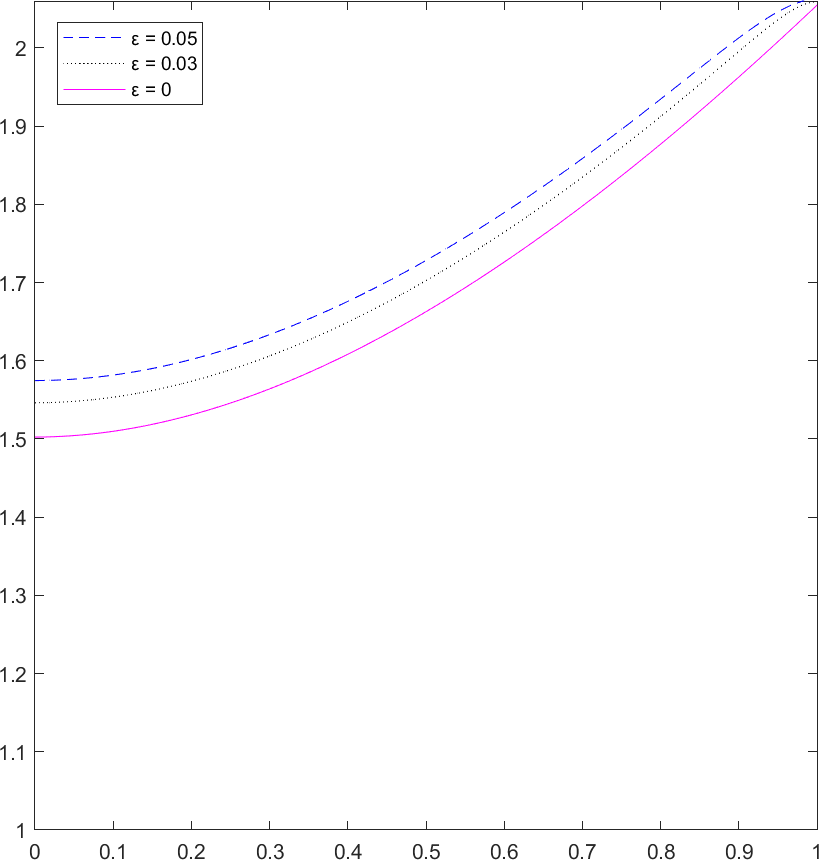}
  \caption{Exact final condition $u_{\eps}(1,\cdot)$ (left) and numerical approximations   of $u_{\eps}(0,\cdot)$ (right)  for $\varepsilon = 0.05$, $\eps=0.03$, and $\eps=0$, with step sizes $\Delta x = 6.25\times 10^{-3}$ and $\Dt=\Delta x /2$.\label{fig:test1_nuovo}}
\end{figure}
\subsection{Nonlinear problem on a circular domain}\label{test2}
Let $T=1$, $\OO=\{ x = (x_1,x_2)\in \RR^2 \, |  \; |x | < 1 \}$,   $\sigma(t,x) = \sqrt{2} (\sin(x_1 + x_2), \cos(x_1 + x_2))$, and
$$\ba{rcl}f(t,x) &=&   \left(\frac{1}{2} - t\right)\sin (x_1) \sin(x_2) + \left( \frac{3}{2} - t \right) \bigg( \sqrt{\cos^2(x_1) \sin^2(x_2) + \sin^2(x_1) \cos^2(x_2)} \\[6pt]
\; & \; &  - 2 \sin(x_1 + x_2) \cos(x_1 + x_2) \cos(x_1) \cos(x_2)\bigg),\\[6pt]
g(t,x) &=&  \left(\frac{3}{2} - t \right) \left( x_1\cos(x_1)\sin(x_2)+ x_2 \sin(x_1)\cos(x_2)\right).
\ea$$ 
Then $\ov{\OO}_{T} \ni (t,x_1,x_2)\mapsto \bar{u}(t,x_1,x_2) = \left(\frac{3}{2} - t \right) \sin(x_1)\sin(x_2)$ is the unique classical solution to 
\be\label{eq:test_2}
\ba{rcl}
\partial_t u - \half \text{Tr}(\sigma\sigma^{\top}D^2u)+ |Du| &=& f \quad \text{in } \OO_{T}, \\[6pt]
   \langle n, Du\rangle &=& g  \quad \text{in } [0,T)\times\partial \OO, \\[6pt]
   u(0,x) & =& \bar{u}(0,x) \quad \text{in } x\in \ov{\OO}.
\ea
\ee

In Figure~\ref{fig:test2}, we show the numerical solution at the final time $T=1$ computed on an unstructured triangular mesh  $\G_{\Delta x}$ with mesh size $\Delta x=1.25 \cdot 10^{-1}$. On the left, we plot the result together with the contour lines. On the right, we plot the approximation together with  the mesh used to compute it. 
\begin{figure}[h!]
\centering
        \includegraphics[width=0.5\textwidth]{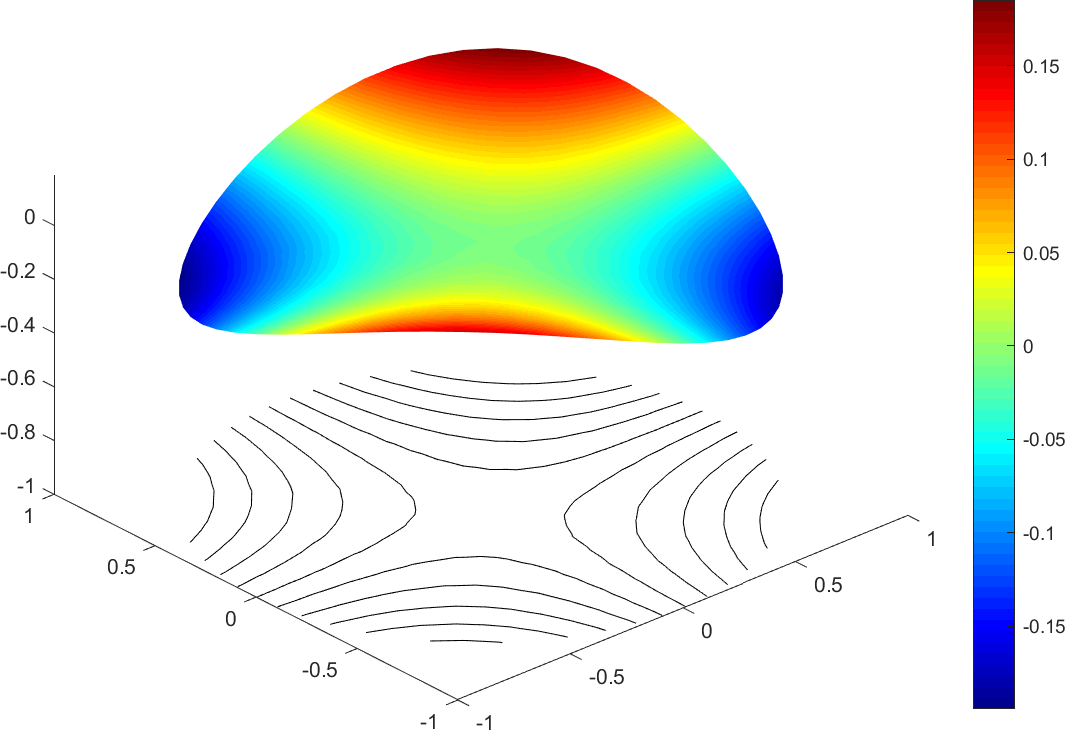}
         \includegraphics[width=0.4\textwidth]{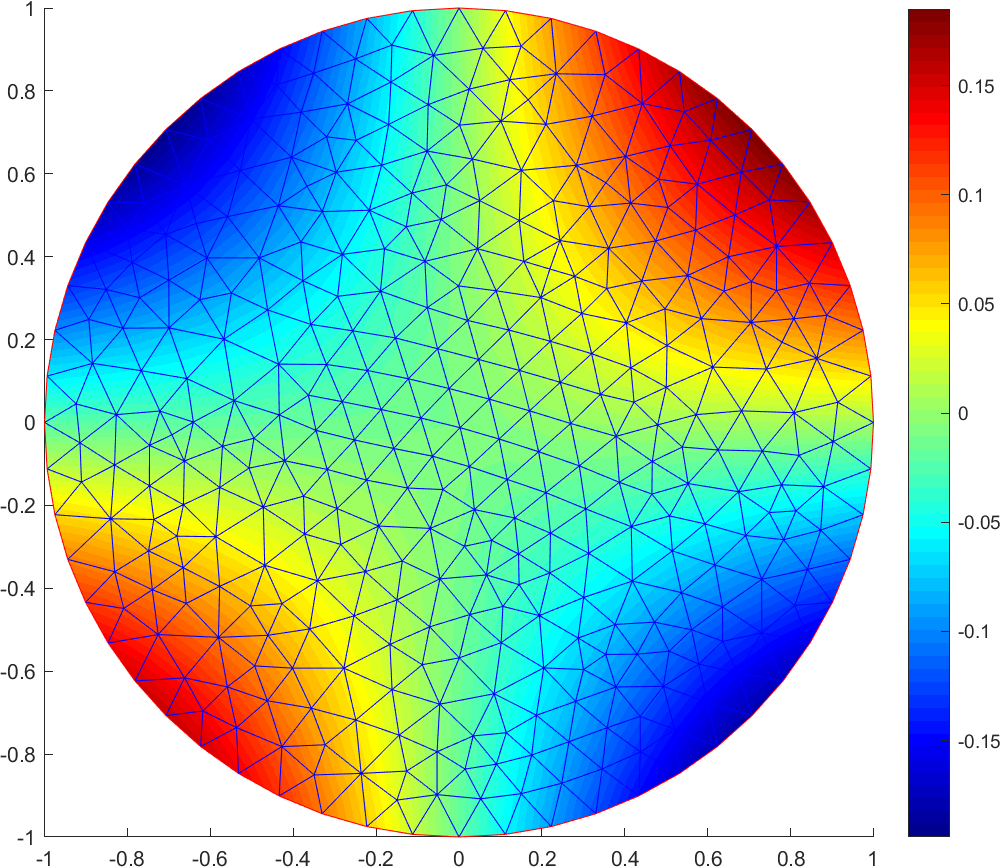}
	\caption{Numerical solution at time $T=1$ of problem in subsect.\ref{test2} with  Neumann boundary condition, computed  with $\Delta x = 0.125$ and $\Dt = \Delta x / 2$.}\label{fig:test2}
\end{figure}\\

Given an element $\hat{T}$ of the triangulation, we denote by  $x_{\hat T}$  its barycenter and by $|\hat{T}|$ its area. We show  in Tables 3 and 4 the errors  
\be\label{test:errors_second_example} E_\infty=\underset {i\in \I_{\Dx}}\max |U_{N_T,i}-\bar{u}(t_{N_T},x_i)|,\quad 
         E_{1}=\sum_{\hat T \in \T_{\Delta x}} |\hat T| \big|I[U_{N_T,(\cdot)}](x_{\hat T})-\bar{u}(t_{N_T},x_{\hat T})\big|,
\ee
and the corresponding convergence rates $p_{\infty}$ and $p_1$. In each table, we specify in the first column the mesh size $\Delta x$. To obtain the results shown in Tables 3 and 4, we have chosen  $\bar{c}$ in \eqref{definizione_R_n} and \eqref{definition_c} as $\bar c=0.25$ and $\bar c = 0.5$, respectively. For both choices of $\bar c$, we observe similar errors and an analogue behavior of the convergence rates. As in the previous example, an order of convergence close to $1$ is obtained.
 
\begin{table}[h!]
\caption{Errors and convergence rates for the approximation of \eqref{eq:test_2} with $\bar c=0.25$.
\label{table:testcontrols}}
  \centering
\begin{tabular}{lllllllll} 
\hline\noalign{\smallskip}
& \multicolumn{4}{c|}{$\Dt = \Delta x$} & \multicolumn{4}{c}{$\Dt = \Delta x/2$} \\
\noalign{\smallskip}\hline\noalign{\smallskip}
 $\Delta x$& $E_{\infty}$ &  $E_1$ & $p_{\infty}$ &  $p_1$ & $E_{\infty}$ &  $E_1$ & $p_{\infty}$ &  $p_1$    \\
\hline
$2.50 \cdot 10^{-1}$ &  $2.73\cdot 10^{-1}$ & $2.95\cdot 10^{-1}$  & -  &  - & $1.22\cdot 10^{-1}$ & $1.07\cdot 10^{-1}$  & -  &  -  \\
\hline 
$1.25 \cdot 10^{-1}$ &  $1.24\cdot 10^{-1}$ & $1.12\cdot 10^{-1}$  & 1.14 & 1.40  & $5.54\cdot 10^{-2}$ & $4.57\cdot 10^{-2}$  & 1.14 &  1.24 \\
\hline
$6.25 \cdot 10^{-2}$ & $5.55\cdot 10^{-2}$ & $4.72\cdot 10^{-2}$  & 1.16  &  1.24 & $2.39\cdot 10^{-2}$ & $2.11\cdot 10^{-2}$  & 1.21 &  1.11  \\
\hline 
$3.125 \cdot 10^{-2}$ &  $2.49\cdot 10^{-2}$ & $2.16\cdot 10^{-2}$  & 1.16  & 1.13 & $1.22\cdot 10^{-2}$ & $1.10\cdot 10^{-2}$  & 0.97 &  0.94  \\
\hline 
 \end{tabular}
\end{table}
\begin{table}[h!]
  \centering
  \caption{Errors and convergence rates for the approximation of \eqref{eq:test_2} with $\bar c=0.5$.
\label{table:testcontrols_2}}
\begin{tabular}{lllllllll} 
\hline\noalign{\smallskip}
& \multicolumn{4}{c|}{$\Dt = \Delta x$} & \multicolumn{4}{c}{$\Dt = \Delta x/2$} \\
\noalign{\smallskip}\hline\noalign{\smallskip}

$\Delta x$& $E_{\infty}$ &  $E_1$ & $p_{\infty}$ &  $p_1$ & $E_{\infty}$ &  $E_1$ & $p_{\infty}$ &  $p_1$    \\
\hline
$2.50 \cdot 10^{-1}$ &  $2.65\cdot 10^{-1}$ & $2.55\cdot 10^{-1}$  & -  &  - & $1.18\cdot 10^{-1}$ & $1.02\cdot 10^{-1}$  & -  &  -  \\
\hline 
$1.25 \cdot 10^{-1}$ &  $1.23\cdot 10^{-1}$ & $1.12\cdot 10^{-1}$  & 1.11 & 1.19  & $5.60\cdot 10^{-2}$ & $4.72\cdot 10^{-2}$  & 1.08 &  1.11 \\
\hline
$6.25 \cdot 10^{-2}$ & $5.74\cdot 10^{-2}$ & $5.06\cdot 10^{-2}$  & 1.10  &  1.15 & $2.64\cdot 10^{-2}$ & $2.27\cdot 10^{-2}$  & 1.08 &  1.06  \\
\hline 
$3.125 \cdot 10^{-2}$ &  $2.70\cdot 10^{-2}$ & $2.39\cdot 10^{-2}$  & 1.09  & 1.08 & $1.22\cdot 10^{-2}$ & $1.10\cdot 10^{-2}$  & 1.11 &  1.05 \\
\hline 
 \end{tabular}

\end{table}

Next, we consider the same problem but with  oblique boundary conditions.
More precisely, for  $x = (x_1,x_2) \in \partial \OO$  we set
$$
\gamma(x) = \left(x_1\cos(\pi/6) + x_2 \sin(\pi/6),x_2\cos(\pi/6) - x_1 \sin(\pi/6)\right)$$ and 
$$
\ba{rcl}
\tilde{g}(t,x)&=& \left(\frac{3}{2} - t \right) \big[\left(x_1\cos(\pi/6) + x_2 \sin(\pi/6)\right)\cos(x_1)\sin(x_2) \\[6pt]
\; & \; & + \left(x_2\cos(\pi/6) - x_1 \sin(\pi/6)\right) \sin(x_1)\cos(x_2)\big] \quad \text{in } [0,T)\times \partial \OO.
\ea
$$
Then $\bar{u}$ is the unique classical solution to 
\be\label{eq:test_2_bis}
\ba{rcl}
\partial_t u - \half \text{Tr}(\sigma\sigma^{\top}D^2u)+ |Du| &=& f \quad \text{in } \OO_{T}, \\[6pt]
   \langle \gamma, Du\rangle &=& \tilde{g}  \quad \text{in } [0,T)\times\partial \OO, \\[6pt]
   u(0,x) & =& \bar{u}(0,x) \quad \text{in } x\in \ov{\OO}.
\ea
\ee
The   solution  $\bar{u}$  is approximated by using the same unstructured meshes as in the previous case.   We show in Tables \ref{table:testcontrolsgammac025} and \ref{table:testcontrolsgammac05} the errors \eqref{test:errors_second_example}  computed with  $\bar c=0.25$ and $\bar c = 0.5$,  respectively. As in the previous case,  we observe similar errors and an analogue behavior of the convergence rates for both choices of $\bar c$.   We also observe a slight degradation of the errors and the convergence rates in the more complicated case of oblique boundary conditions. 
\begin{table}[h!]
  \centering
  \caption{Errors and convergence rates for the approximation of \eqref{eq:test_2_bis} with $\bar c=0.25$
\label{table:testcontrolsgammac025}}

\begin{tabular}{lllllllll} 
\hline\noalign{\smallskip}
& \multicolumn{4}{c|}{$\Dt = \Delta x$} & \multicolumn{4}{c}{$\Dt = \Delta x/2$} \\
\noalign{\smallskip}\hline\noalign{\smallskip}
$\Delta x$ & $E_{\infty}$ &  $E_1$ & $p_{\infty}$ &  $p_1$ & $E_{\infty}$ &  $E_1$ & $p_{\infty}$ &  $p_1$    \\
\hline
$2.50 \cdot 10^{-1}$ &  $3.06\cdot 10^{-1}$ & $4.38\cdot 10^{-1}$  & -  &  - & $1.50\cdot 10^{-1}$ & $2.08\cdot 10^{-1}$  & -  &  -  \\
\hline 
$1.25 \cdot 10^{-1}$ &  $1.56\cdot 10^{-1}$ & $2.25\cdot 10^{-1}$  & 0.97& 0.96  & $7.96\cdot 10^{-2}$ & $1.17\cdot 10^{-1}$  & 0.91 &  0.83 \\
\hline
$6.25 \cdot 10^{-2}$ & $8.10\cdot 10^{-2}$ & $1.21\cdot 10^{-1}$  & 0.95  &  0.89 & $4.36\cdot 10^{-2}$ & $6.84\cdot 10^{-2}$  & 0.88 &  0.77 \\
\hline 
$3.125 \cdot 10^{-2}$ &  $4.47\cdot 10^{-2}$ & $7.17\cdot 10^{-2}$  & 0.86  & 0.75 & $2.58\cdot 10^{-2}$ & $4.26\cdot 10^{-2}$  & 0.76 &  0.68  \\
\hline 
 \end{tabular}
\end{table}

\begin{table}[h!]
  \centering
  \caption{Errors and convergence rates for the approximation of \eqref{eq:test_2_bis} with $\bar c=0.5$.}
\label{table:testcontrolsgammac05}
\begin{tabular}{lllllllll} 
\hline\noalign{\smallskip}
& \multicolumn{4}{c|}{$\Dt = \Delta x$} & \multicolumn{4}{c}{$\Dt = \Delta x/2$} \\
\noalign{\smallskip}\hline\noalign{\smallskip}
$\Delta x$  & $E_{\infty}$ &  $E_1$ & $p_{\infty}$ &  $p_1$ & $E_{\infty}$ &  $E_1$ & $p_{\infty}$ &  $p_1$    \\
\hline
$2.50 \cdot 10^{-1}$ &  $2.94\cdot 10^{-1}$ & $3.81\cdot 10^{-1}$  & -  &  - & $1.42\cdot 10^{-1}$ & $1.69\cdot 10^{-1}$  & -  &  -  \\
\hline 
$1.25 \cdot 10^{-1}$ &  $1.49\cdot 10^{-1}$ & $1.88\cdot 10^{-1}$  & 0.98& 1.02  & $7.22\cdot 10^{-2}$ & $8.56\cdot 10^{-2}$  & 0.98 &  0.98 \\
\hline
$6.25 \cdot 10^{-2}$ & $7.55\cdot 10^{-2}$ & $9.33\cdot 10^{-2}$  & 0.98  &  1.01 & $3.79\cdot 10^{-2}$ & $4.63\cdot 10^{-2}$  & 0.93 &  0.89 \\
\hline 
$3.125 \cdot 10^{-2}$ &  $3.95\cdot 10^{-2}$ & $5.02\cdot 10^{-2}$  & 0.93  & 0.89& $2.12\cdot 10^{-2}$ & $2.75\cdot 10^{-2}$  & 0.84 &  0.75  \\
\noalign{\smallskip}\hline 
\end{tabular}
\end{table}

\subsection{Nonlinear problem on a non-smooth domain with mixed Dirichlet-Neumann boundary conditions}\label{test3}
In this last example, we deal with a problem of exiting from a bounded rectangular domain with an circular obstacle inside of it. We model this problem by considering  a modification of \eqref{HJB_principal_equation} including mixed Dirichlet-Neumann boundary conditions, with a large time horizon $T$ in order to reach a stationary solution. We consider the space domain  
$$\OO = \bigg( (-1,1) \times (-0.5,0.5)\bigg)\setminus \{ x \in \RR^2 \, | \,  | x - (-0.5, 0)|\leq 0.2 \},$$
a control set $A = \{a \in \RR^2 \, | \,  | a | = 1 \}$, a drift $\mu(t,x,a) = a$, a diffusion coefficient $\sigma(t,x,a) = 0.1 I_{2}$, where $I_2$ is the identity matrix of size $2$, 
a running cost  $f\equiv 1$, and an initial condition  $\Psi\equiv 0$. We impose   constant Dirichlet boundary conditions on some parts of $\partial \OO$, representing the exits of the domain,  in order to model some exit costs. More precisely,   Dirichlet boundary conditions (or exit costs) $u=0$ and $u=0.2$ are imposed on  $\partial \OO_1 = \{ x=(x_1,x_2) \in \partial \OO \, | \,  x_1 = -1, | x_2 | \leq 0.2 \}$ and $\partial \OO_2 = \{ x=(x_1,x_2) \in \partial \OO \, | \,  x_1 = 1, | x_2 | \leq 0.2 \}$, respectively.  We also consider homogeneous Neumann boundary conditions on the remaining part of the boundary.

We treat the Dirichlet boundary conditions by using an extrapolation technique. This approximation has been proposed in \cite{BCCF2020} and has been shown to be more accurate with respect to the methods  proposed in \cite{milstein:2001,bonaventura:2018}.
 We show in Figure \ref{fig:test3} the numerical approximation computed on an unstructured mesh  with mesh size $\Dx = 0.01$, a time step $\Dt = \Dx$ and final time $T = 3$. Figure \ref{fig:test3_2} diplays the quiver plot of $-Du$ at time $T=3$. 
\begin{figure}[h!]
        \includegraphics[width=0.5\textwidth]{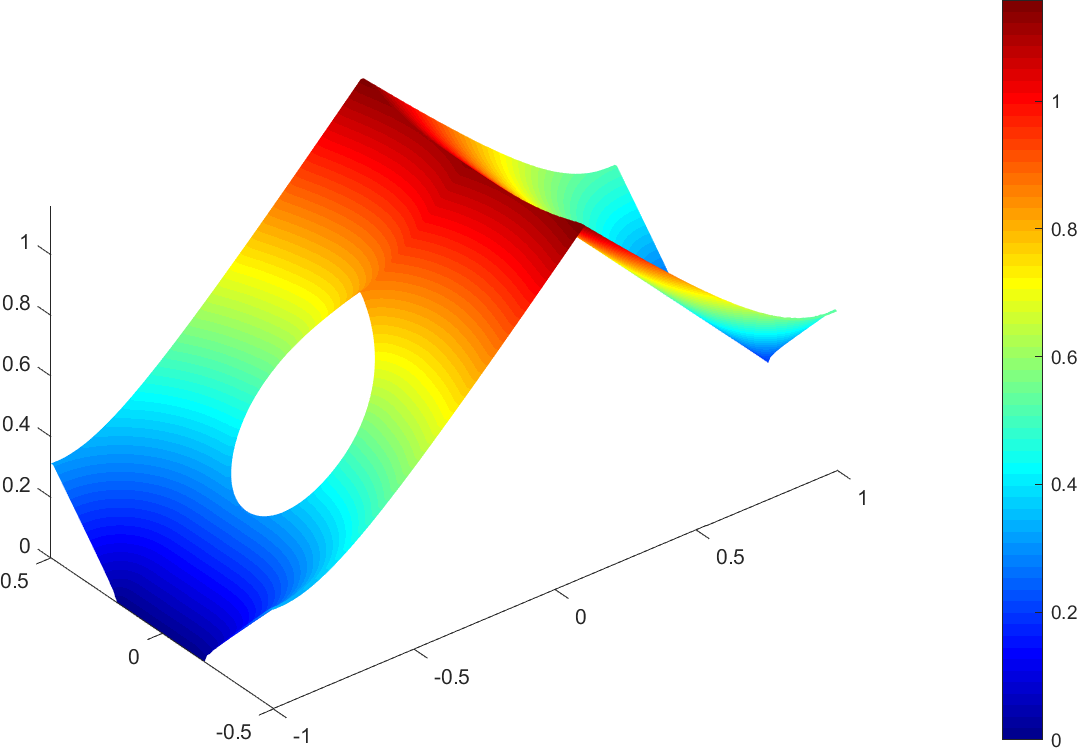}
         \includegraphics[width=0.5\textwidth]{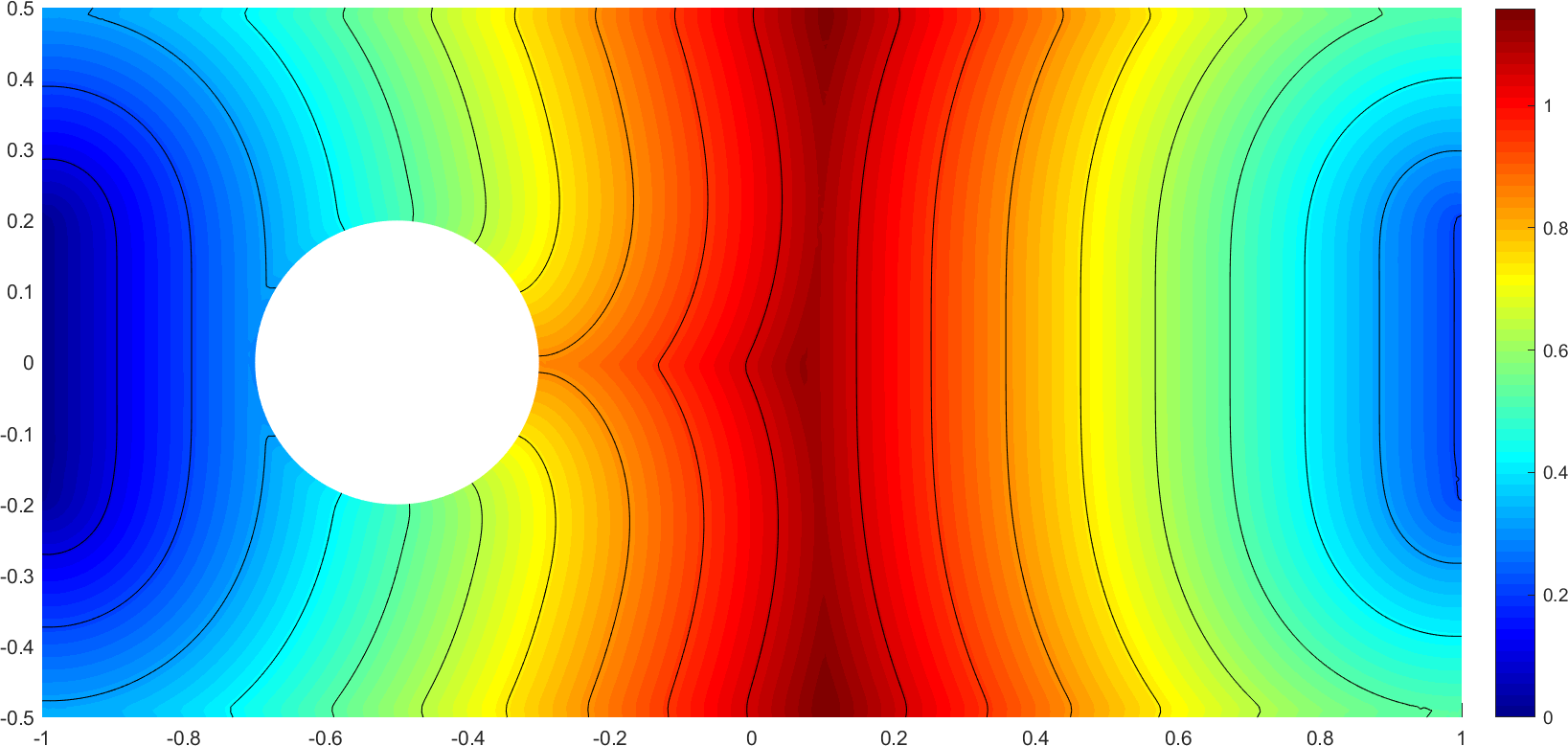}
 	\caption{Solution at time $T=3$ with $\Delta x = 0.01$ and with $\Dt = \Dx$.}\label{fig:test3}
\end{figure}

\begin{figure}[h!]
\centering
  \includegraphics[width=\textwidth]{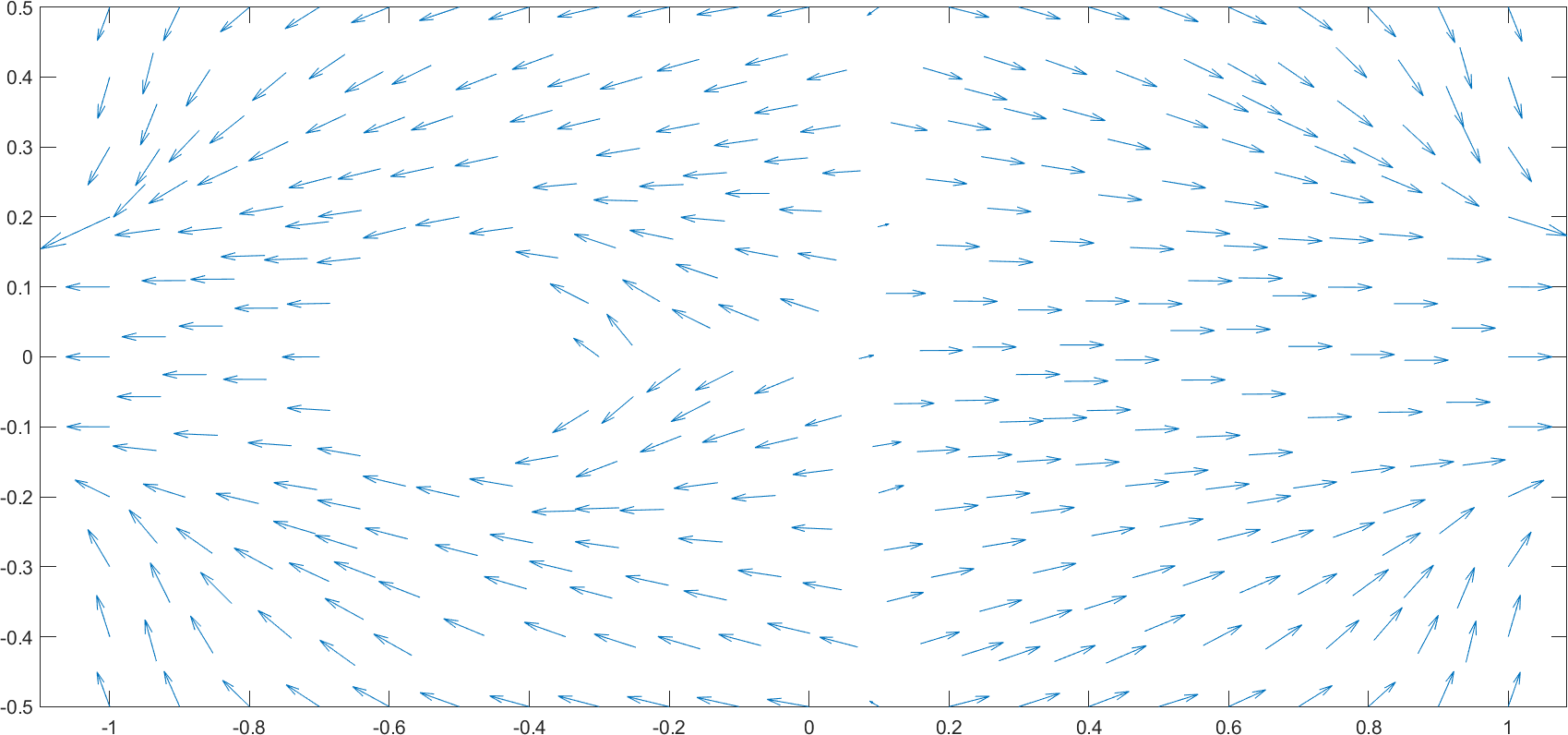}
\caption{Quiver plot of $-Du$ at time $T=3$.}\label{fig:test3_2}
\end{figure}

\bibliographystyle{plain}
\bibliography{bibFP}
\section{Appendix}
In this appendix we first study the existence of the projection of $x$  onto $\partial \OO$ parallel to $\gamma_b$ in a neighborhood of $\partial \OO$  and for  $b\in B$. These projections play an important role in the construction of our scheme in Sect.~\ref{fully_discrete_scheme}. The  following result is an extension of a result in \cite[Section~1.2]{gobet2001} to the regularity that we assume in this paper and, more importantly, to the dependence of $\gamma$ on $b$.   Recall that in {\bf(H3)}  $\partial\OO$ is assumed to be of class $C^3$. However, the result in Proposition \ref{oblique_projection_well_defined} below is also valid if $\partial\OO$ is only of class $C^2$. 
\begin{proposition}\label{oblique_projection_well_defined}
	There exists $R > 0$ such that, for any $x\in\RR^{N}$ satisfying $d(x,\partial \OO)  < R$ and for any $b \in B$, there exist a unique $p^{\gamma_b}(x) \in \partial \OO$ and a unique $d^{\gamma_b}(x) \in \RR$ such that 
	\begin{equation} \label{eq:def_proj-dist}
		x = p^{\gamma_b}(x) + d^{\gamma_b}(x) \gamma_{b}(p^{\gamma_b}(x)) .
	\end{equation}
	The mappings $(x,b)\mapsto p^{\gamma_b}(x)$ and $(x,b)\mapsto d^{\gamma_b}(x)$, called respectively the projection onto $\partial \OO$ parallel to $\gamma_b$
	and the algebraic distance to $\partial \OO$ parallel to $\gamma_b$, are of class $C^1$.
\end{proposition}

\begin{proof} We use the same outline and, as much as possible, the same notations than those in \cite{gobet2001}. 
	
Let us fix $(s,b_0) \in \partial \OO \times B$. Let $g^s \colon U^s \to \partial \OO$ be a $C^2$ parameterization of $\partial \OO$ in a neighborhood of $s$, with $U^s$ being an open subset of $\RR^{N-1}$, $z_0 \in U^s$, and $g^s(z_0) = s$. By {\bf(H3)} the function 
$$  U^s \times \RR \times \V \ni (z,\lambda,b)   \mapsto  G^s(z,\lambda,b) = ( g^s(z) + \lambda \gamma_{b}(g^s(z)),b) \in \RR^{N}\times \RR^{N_{B}}$$
is of class $C^1$.  The Jacobian matrix of $G^s$ has the form
$$
	 J^{s}(z,\lambda,b) = 
	\left(\begin{array}{@{}c|c@{}}
	J_{z,\lambda}(z, \lambda,b)
	& J_{b}(z, \lambda,b) \\
	\hline
	0_{N_B,N} & I_{N_B}
	\end{array}\right),
$$
	where $J_{z,\lambda}(z, \lambda,b)$ coincides with $J(z, \lambda)$ of the Appendix~A of \cite{gobet2001}, that is
	\[
	J_{z,\lambda}(z,\lambda,b) = 
	\left(\begin{array}{@{}c|c|c|c@{}}
	&&&\\
	\partial_{z_1}g^s(z) + \lambda \partial_{z_1}\gamma_b(g^s(z)) & \cdots 
	& \partial_{z_{N-1}}g^s(z) + \lambda \partial_{z_{N-1}}\gamma_b(g^s(z))
	& \gamma_b(g^s(z)) \\
	&&&\\
	\end{array}\right).
	\]
	In particular, for $\lambda = 0$,
	\[
	J_{z,\lambda}(z,0,b) = 
	\left(\begin{array}{@{}c|c|c|c@{}}
	&&&\\
	\partial_{z_1}g^s(z) & \cdots & \partial_{z_{N-1}}g^s(z) & \gamma_b(g^s(z)) \\
	&&&\\
	\end{array}\right)
	\]
	is invertible since its $N-1$ first columns span the tangent space to $\partial \OO$ at $g^s(z)$
	and, since 
	$$\langle n(g^s(z)), \gamma_b(g^s(z)) \rangle > 0,$$ 
	its last column is non tangent to $\partial \OO$.
	It follows that $ J^{s}(z,0,b)$ is also invertible, and we can therefore apply the inverse mapping theorem to
	$G^s$ at $(z_0,0,b_0)$ to obtain the existence of a neighborhood $V^{s,b_0}$ of $(s,b_0)$ and $C^1$ mappings $V^{s,b_0} \ni (x,b) \mapsto p^{\gamma_b}(x)\in \partial \OO$ and $V^{s,b_0} \ni (x,b)\mapsto d^{\gamma_b}(x)$ such that \eqref{eq:def_proj-dist} holds for every $(x,b)\in V^{s,b_0}$. The compactness of $\partial \OO \times B \subset \cup_{(s,b_0) \in \partial \OO \times B}V^{s,b_0}$
	enables to consider a finite number of $(s_i,(b_0)_i)$, $1\le i \le k$,
	such that $\partial \OO \times B \subset \cup_{i=1}^k V^{s_i,(b_0)_i}$. Then there exists $\bar R>0$ such that $\{ y \in \mathbb R^N \; | \;   d(y,\partial \OO) < \bar R \} \times B \subset \cup_{i=1}^k V^{s_i,(b_0)_i}$.
	In particular for any $x$ such that $d(x,\partial \OO) < \bar R$ and any $b \in B$, there exist a least a point $p^{\gamma_b}(x)$ and a scalar $d^{\gamma_b}(x)$ such that \eqref{eq:def_proj-dist} holds.
	We claim that there exists $R \in (0,\bar R)$ such that for any $x$ satisfying $d(x,\partial \OO) < R$ and any $b \in B$, $p^{\gamma_b}(x)$ is unique (and as a consequence $d^{\gamma_b}(x)$ is also unique). 
	Assume that this is not the case. Then (considering for example $R =\frac{1}{k}$) one can build a sequence $(x_k, b_k)_{k\in \NN}$ converging (after extraction a subsequence) to some point $(\hat{s},\hat{b}) \in \partial \OO \times B$ 
	and such that for all $k\in \NN$,
	$x_k$ has two distinct projections $p^{\gamma_{b_k}}_i(x_k)$ with associated algebraic distances $d^{\gamma_{b_k}}_i(x_k)$, $i=1,2$.
	At the limit point $\hat{s}$, we consider $G^{\hat{s}}$ which is a local diffeomorphism on a neighborhood of $(\hat{z},0,\hat{b})$ (with $g^{\hat{s}}(\hat{z})=\hat{s}$).
	Since $x_k \to \hat{s} \in \partial \OO$, then $p^{\gamma_{b_k}}_i(x_k) \to \hat{s}$ and $d^{\gamma_{b_k}}_i(x_k) \to 0$, $i=1,2$.
	Let $z_{i,k}$ be such that $g^{\hat{s}}(z_{i,k}) = p^{\gamma_{b_k}}_i(x_k)$ and $\lambda_{i,k} = d^{\gamma_{b_k}}_i(x_k)$, $i=1,2$.
	Then $(z_{i,k},\lambda_{i,k},b_k)_k$, $i=1,2$, are distinct sequences that both converge to $(\hat{z},0,\hat{b})$ and have the same image $G^{\hat{s}}(z_{i,k},\lambda_{i,k},b_k) = (x_k,b_k)$. This contradicts that $G^{\hat{s}}$ is a local diffeomorphism on a neighborhood of $(\hat{z},0,\hat{b})$.	
\end{proof}

For any $\eps \geq 0$ let us define
\begin{align}
D_\eps = \{ x \in \overline \OO \; | \; d(x, \partial \OO) > \eps \},\label{eq:internalset}\\
\partial D_\eps = \{ x \in \overline \OO  \; | \; d(x, \partial \OO)  =  \eps\},\label{partial_o_delta}\\
L_\eps =  \{ x \in \overline \OO \; | \; d(x, \partial \OO)  \le  \eps \}.\label{layer}
\end{align}

Now we focus on the existence of projections of $x\in L_{\eps}$ onto $\partial D_{\eps}$ and the regularity of $L_{\eps} \ni x \mapsto d(x,D_{\eps}) \in \RR$. These results are important in order to show Lemma \ref{lemma:milstein_stab_d} which is the key to obtain the stability of the scheme in Proposition \ref{stability}.

\begin{lemma} \label{lemma:dist1}
The following hold:
\begin{enumerate}
\item[{\rm(i)}]  There exists $\eta>0$ such that  on $L_\eta$, the projection $p_{\partial \OO}$ onto $\partial \OO$  is well-defined and $C^1$.
\item[{\rm(ii)}]  The distance function $L_\eta \ni x \mapsto d(x, \partial \OO)\in \RR$ is $C^3$, and $Dd(\cdot, \partial \OO)(x) = - n(p_{\partial \OO}(x))$.
\end{enumerate}
Let $\delta \in [0,\eta]$. Then the following hold:
\begin{enumerate}
\item[{\rm(iii)}]  $\partial D_\delta$ is of class $C^3$ and, denoting by $n_\delta(x)$ the unit outward normal at $x\in \partial D_\delta$, we have $n_{\delta}(x)=n(p_{\partial \OO}(x))$. 
\item[{\rm(iv)}] For every $x\in L_\delta$, $p=p_{\partial \OO}(x)-\delta n(p_{\partial \OO}(x))$ is a projection of $x$ onto $\partial D_\delta$. 
\item[{\rm(v)}] The function $x \mapsto d(x, \partial D_\delta)$ is of class $C^3$ on $L_\delta$ and $d(x,\partial \OO)+ d(x,\partial D_{\delta})=\delta$ for every $x\in L_{\delta}$. 
\end{enumerate}
\end{lemma}
\begin{proof}{\rm(i)}\&{\rm(ii)} See \cite[Lemma~14.16]{MR1814364}.\smallskip\\
{\rm(iii)} This follows  from {\rm(ii)} and \eqref{partial_o_delta}. \smallskip\\
{\rm(iv)}\&{\rm(v)} Let us first show that $p\in \partial D_{\delta}$. We have $d(p,\partial \OO)\leq |p -p_{\partial \OO}(x)|= \delta$. Thus, $p\in L_{\delta}$ and, by {\rm(i)},  $p_{\partial \OO}(x)=p_{\partial \OO}(p)$, which implies that $d(p,\partial \OO)= \delta$ and hence $p\in \partial D_\delta$. Since 
$$x = p_{\partial \OO}(x) - d(x,\partial \OO) n(p_{\partial \OO}(x)),$$
we obtain $d(x,\partial D_\delta)\leq |p-x|=\delta-d(x,\partial \OO)$. Assume that $d(x, \partial D_\delta) < \delta - d(x, \partial \OO)$. Then there exists $p' \in \partial D_\delta$ such that $| x - p' | < \delta - d(x, \partial \OO)$. This implies that
$$
	\delta=d(p',\partial \OO) \le  | p' - p_{\partial \OO}(x)   | \le  | p' - x  | +  |x - p_{\partial \OO}(x)   | <  \delta,
$$
which is impossible. Thus 	
$$|p-x|=d(x, \partial D_\delta) = \delta - d(x, \partial \OO).$$
The first equality above implies that $p$ is a projection of $x$ onto $\partial D_\delta$. Since $x\in L_{\delta}$ is arbitrary, the second equality above and {\rm(ii)} imply that {\rm(v)} holds. 
\end{proof}
\end{document}

%% file: mymacro.tex
\newcommand{\Dt}{{\Delta  t}} 
\newcommand{\Dx}{{\Delta  x}}

\newcommand{\ov}[1]{\overline{#1}}

\def\weight(#1,#2){c_{#1,#2}}

\def\ra{{\rm ra}}

 \if {
  
 } \fi

\if {

}{{\caption}}
} \fi

\def\G{\mathcal{G}}
\def\H{\mathcal{H}}
\def\I{\mathcal{I}}
\def\J{\mathcal{J}}

\def\L{\mathcal{L}}

\def\SS{\mathcal{S}}
\def\T{\mathcal{T}}

\def\V{\mathcal{V}}

\def\eps{\varepsilon}

\newcommand{\floor}[1]{\lfloor #1 \rfloor}

\def\half{\mbox{$\frac{1}{2}$}}

\def\1B{{\bf  1}}

\newcommand{\NN}{\mathbb{N}}

\newcommand{\OO}{\mathcal{O}}
\newcommand{\RR}{\mathbb{R}}

\def\II{\mathbb{I}}
\def\EE{\mathbb{E}}
\def\PP{\mathbb{P}}

\newcommand\be{\begin{equation}}
\newcommand\ee{\end{equation}}
\newcommand\ba{\begin{array}}
\newcommand\ea{\end{array}}
\newcommand{\bean}{\begin{eqnarray*}}
\newcommand{\eean}{\end{eqnarray*}}

\def\ds{\displaystyle}

\def\la{\langle}
\def\ra{\rangle}


%% file: neumann_arxiv.bbl
\begin{thebibliography}{10}

\bibitem{MR2034614}
R.~Abgrall.
\newblock Numerical discretization of boundary conditions for first order
  {H}amilton-{J}acobi equations.
\newblock {\em SIAM J. Numer. Anal.}, 41(6):2233--2261, 2003.

\bibitem{AF12}
Y.~Achdou and M.~Falcone.
\newblock A semi-lagrangian scheme for mean curvature motion with nonlinear
  neumann conditions.
\newblock {\em Interfaces Free Bound.}, 14(4):455--485, 2012.

\bibitem{BardiCapuzzo96}
M.~Bardi and I.~Capuzzo Dolcetta.
\newblock {\em Optimal control and viscosity solutions of
  {H}amilton-{J}acobi-{B}ellman equations}.
\newblock Birkauser, 1996.

\bibitem{Barles1993}
G.~Barles.
\newblock Fully nonlinear {N}eumann type boundary conditions for second-order
  elliptic and parabolic equations.
\newblock {\em J. Differential Equations}, 106(1):90--106, 1993.

\bibitem{MR1685618}
G.~Barles.
\newblock Nonlinear {N}eumann boundary conditions for quasilinear degenerate
  elliptic equations and applications.
\newblock {\em J. Differential Equations}, 154(1):191--224, 1999.

\bibitem{MR1090787}
G.~Barles and P.-L. Lions.
\newblock Fully nonlinear {N}eumann type boundary conditions for first-order
  {H}amilton-{J}acobi equations.
\newblock {\em Nonlinear Anal.}, 16(2):143--153, 1991.

\bibitem{BS91}
G.~Barles and P.~E. Souganidis.
\newblock Convergence of approximation schemes for fully nonlinear second order
  equations.
\newblock {\em Asymptotic Anal.}, 4(3):271--283, 1991.

\bibitem{barrett_elliott_1986}
J.~W. Barrett and C.~M. Elliott.
\newblock Finite element approximation of the {D}irichlet problem using the
  boundary penalty method.
\newblock {\em Numer. Math.}, 49(4):343--366, 1986.

\bibitem{bonaventura:2018}
L.~Bonaventura, R.~Ferretti, and L.~Rocchi.
\newblock A fully semi-{L}agrangian discretization for the 2{D}
  {N}avier-{S}tokes equations in the vorticity--streamfunction formulation.
\newblock 323:132--144, 2018.

\bibitem{BCCF2020}
Luca Bonaventura, Elisa Calzola, Elisabetta Carlini, and Roberto Ferretti.
\newblock Second order fully semi-lagrangian discretizations of
  advection-diffusion-reaction systems.
\newblock {\em J. Sci. Comput.}, 88(1):Paper No. 23, 29, 2021.

\bibitem{MR2421330}
B.~Bouchard.
\newblock Optimal reflection of diffusions and barrier options pricing under
  constraints.
\newblock {\em SIAM J. Control Optim.}, 47(4):1785--1813, 2008.

\bibitem{MR2399437}
M.~Bourgoing.
\newblock Viscosity solutions of fully nonlinear second order parabolic
  equations with {$L^1$} dependence in time and {N}eumann boundary conditions.
\newblock {\em Discrete Contin. Dyn. Syst.}, 21(3):763--800, 2008.

\bibitem{CamFal95}
F.~Camilli and M.~Falcone.
\newblock An approximation scheme for the optimal control of diffusion
  processes.
\newblock {\em RAIRO Mod\'el. Math. Anal. Num\'er.}, 29(1):97--122, 1995.

\bibitem{CFF10}
E.~Carlini, M.~Falcone, and R.~Ferretti.
\newblock Convergence of a large time-step scheme for mean curvature motion.
\newblock {\em Interfaces Free Bound.}, 12(4):409--441, 2010.

\bibitem{ciarlet}
P.~G. Ciarlet and J.-L. Lions, editors.
\newblock {\em Handbook of numerical analysis. {V}ol. {II}}.
\newblock Handbook of Numerical Analysis, II. North-Holland, Amsterdam, 1991.
\newblock Finite element methods. Part 1.

\bibitem{CrandallIshiiLions92}
M.~G. Crandall, H.~Ishii, and P.-L. Lions.
\newblock User's guide to viscosity solutions of second order partial
  differential equations.
\newblock {\em Bull. Amer. Math. Soc. (N.S.)}, 27(1):1--67, 1992.

\bibitem{MR3042570}
K.~Debrabant and E.~R. Jakobsen.
\newblock Semi-{L}agrangian schemes for linear and fully non-linear diffusion
  equations.
\newblock {\em Math. Comp.}, 82(283):1433--1462, 2013.

\bibitem{DeckelnickHinze2007}
K.~Deckelnick and M.~Hinze.
\newblock Convergence of a finite element approximation to a state-constrained
  elliptic control problem.
\newblock {\em SIAM J. Numer. Anal.}, 45(5):1937--1953, 2007.

\bibitem{falconeferretilibro}
M.~Falcone and R.~Ferretti.
\newblock {\em Semi-{L}agrangian {A}pproximation {S}chemes for {L}inear and
  {H}amilton-{J}acobi {E}quations}.
\newblock MOS-SIAM Series on Optimization, 2013.

\bibitem{MR3049920}
X.~Feng, R.~Glowinski, and M.~Neilan.
\newblock Recent developments in numerical methods for fully nonlinear second
  order partial differential equations.
\newblock {\em SIAM Rev.}, 55(2):205--267, 2013.

\bibitem{MR1814364}
D.~Gilbarg and N.~S. Trudinger.
\newblock {\em Elliptic partial differential equations of second order}.
\newblock Classics in Mathematics. Springer-Verlag, Berlin, 2001.
\newblock Reprint of the 1998 edition.

\bibitem{gobet2001}
E.~Gobet.
\newblock Efficient schemes for the weak approximation of reflected diffusions.
\newblock volume~7, pages 193--202. 2001.
\newblock Monte Carlo and probabilistic methods for partial differential
  equations (Monte Carlo, 2000).

\bibitem{MR3587374}
K.~Hinderer, U.~Rieder, and M.~Stieglitz.
\newblock {\em Dynamic optimization}.
\newblock Universitext. Springer, Cham, 2016.
\newblock Deterministic and stochastic models.

\bibitem{MR2070626}
H.~Ishii and M.-H. Sato.
\newblock Nonlinear oblique derivative problems for singular degenerate
  parabolic equations on a general domain.
\newblock {\em Nonlinear Anal.}, 57(7-8):1077--1098, 2004.

\bibitem{MR667669}
P.-L. Lions.
\newblock {\em Generalized solutions of {H}amilton-{J}acobi equations},
  volume~69 of {\em Research Notes in Mathematics}.
\newblock Pitman (Advanced Publishing Program), Boston, Mass.-London, 1982.

\bibitem{MR709164}
P.-L. Lions.
\newblock Optimal control of diffusion processes and
  {H}amilton-{J}acobi-{B}ellman equations. {I}. {T}he dynamic programming
  principle and applications.
\newblock {\em Comm. Partial Differential Equations}, 8(10):1101--1174, 1983.

\bibitem{Lions1985}
P.-L. Lions.
\newblock Neumann type boundary conditions for {H}amilton-{J}acobi equations.
\newblock {\em Duke Math. J.}, 52(4):793--820, 1985.

\bibitem{Milstein96}
G.~N. Milstein.
\newblock Application of the numerical integration of stochastic equations for
  the solution of boundary value problems with {N}eumann boundary conditions.
\newblock {\em Teor. Veroyatnost. i Primenen.}, 41(1):210--218, 1996.

\bibitem{milstein:2001}
G.N. Milstein and M.V. Tretyakov.
\newblock Numerical solution of the {D}irichlet problem for nonlinear parabolic
  equations by a probabilistic approach.
\newblock {\em IMA Journal of Numerical Analysis}, 21:887--917, 2001.

\bibitem{MR3653852}
M.~Neilan, A.~J. Salgado, and W.~Zhang.
\newblock Numerical analysis of strongly nonlinear {PDE}s.
\newblock {\em Acta Numer.}, 26:137--303, 2017.

\bibitem{MR1181342}
E.~Rouy.
\newblock Numerical approximation of viscosity solutions of first-order
  {H}amilton-{J}acobi equations with {N}eumann type boundary conditions.
\newblock {\em Math. Models Methods Appl. Sci.}, 2(3):357--374, 1992.

\end{thebibliography}
